%% file: main.tex
\documentclass{article}

\usepackage[T1]{fontenc}
\usepackage{enumerate, amsmath, amsfonts, amssymb, amsthm, dsfont, mathrsfs, wasysym, graphics, graphicx, xcolor, url, hyperref, hypcap, xargs, multicol, pdflscape, multirow, hvfloat, array, ae, aecompl, pifont, mathtools, a4wide, float, blkarray, overpic, nicefrac, stmaryrd, anyfontsize, yfonts, array, tabularx}
\usepackage{CJKutf8}
\usepackage{xargs, bbm, enumerate, paralist}
\usepackage[shortlabels, inline]{enumitem}
\usepackage{pdflscape}
\usepackage[noabbrev,capitalise]{cleveref}
\usepackage[normalem]{ulem}
\usepackage{marginnote}
\usepackage{animate}

\hypersetup{colorlinks=true, citecolor=darkblue, linkcolor=darkblue}
\usepackage[all]{xy}
\usepackage{tikz}
\usepackage{tikz-cd}
\usetikzlibrary{trees, decorations, decorations.pathmorphing, decorations.markings, decorations.shapes, shapes, arrows, matrix, calc, fit, intersections, patterns, angles}
\graphicspath{{figures/}{figures/diagonals/}{figures/walks/}{figures/tubes/}{figures/blocks/}}
\makeatletter\def\input@path{{figures/}}\makeatother
\usepackage{caption}
\usepackage[export]{adjustbox}

\usepackage{setspace}

\usepackage{paralist}
\usepackage{shuffle}

\let\OLDthebibliography\thebibliography
\renewcommand\thebibliography[1]{
  \OLDthebibliography{#1}
  \setlength{\parskip}{2.2pt}
  \setlength{\itemsep}{0pt plus 0.3ex}
}

\DeclareFontEncoding{LY}{}{}
\DeclareFontSubstitution{LY}{yfrak}{m}{n}
\DeclareFontEncoding{LYG}{}{}
\DeclareFontSubstitution{LYG}{ygoth}{m}{n}
\DeclareFontFamily{LYG}{ygoth}{}
\DeclareFontShape{LYG}{ygoth}{m}{n}{<->ygoth}{}
\DeclareFontFamily{LY}{yfrak}{}
\DeclareFontShape{LY}{yfrak}{m}{n}{<->yfrak}{}
\DeclareFontFamily{LY}{ysmfrak}{}
\DeclareFontShape{LY}{ysmfrak}{m}{n}{<->ysmfrak}{}
\DeclareFontFamily{LY}{yswab}{}
\DeclareFontShape{LY}{yswab}{m}{n}{<->yswab}{}

\newtheorem{thmUniv}{Theorem}

\newtheorem{theorem}{Theorem}[section]
\newtheorem{corollary}[theorem]{Corollary}
\newtheorem{proposition}[theorem]{Proposition}
\newtheorem{lemma}[theorem]{Lemma}
\newtheorem{conjecture}[theorem]{Conjecture}
\newtheorem*{theorem*}{Theorem}

\theoremstyle{definition}
\newtheorem{definition}[theorem]{Definition}
\newtheorem{example}[theorem]{Example}
\newtheorem{remark}[theorem]{Remark}

\crefname{equation}{Equation}{Equations}

\newcommand{\R}{\mathbb{R}} 

\renewcommand{\c}[1]{{\mathcal{#1}}} 
\renewcommand{\b}[1]{{\boldsymbol{#1}}} 
\newcommand{\go}[1]{{\textgoth{#1}}} 

\renewcommand{\emptyset}{\varnothing} 
\renewcommand{\epsilon}{\varepsilon} 

\newcommand{\ssm}{\smallsetminus} 
\newcommand{\one}{{1\!\!1}} 
\newcommandx{\ones}[1][1=n]{\one_{#1}} 

\DeclareMathOperator{\conv}{conv} 
\DeclareMathOperator{\cone}{cone} 

\newcommand{\ie}{\textit{i.e.}~} 
\newcommand{\eg}{\textit{e.g.}~} 
\definecolor{darkblue}{rgb}{0,0,0.7} 
\definecolor{green}{RGB}{57,181,74} 
\definecolor{violet}{RGB}{147,39,143} 
\newcommand{\darkblue}{\color{darkblue}} 
\newcommand{\defn}[1]{\textsl{\darkblue #1}} 

\makeatletter
\def\part{\@startsection{part}{1}%
\z@{.7\linespacing\@plus\linespacing}{.8\linespacing}%
{\LARGE\sffamily\centering}}
\makeatother

\newcommandx{\restrG}[2][1=\mu,2=G]{{#2}\left|_{#1}\right.}
\newcommandx{\contrG}[2][1=\mu,2=G]{{#2}\left/_{#1}\right.}

\newcommand{\polytope}[1]{\mathsf{#1}}
\newcommand{\polytopeP}{\mathsf{P}}

\newcommand{\polytopeQ}{\mathsf{Q}}
\newcommand{\polytopeF}{\mathsf{F}}
\newcommand{\polytopeH}{\mathsf{H}}
\newcommand{\polytopeC}{\mathsf{C}}

\newcommandx{\PivotPolytope}[2][1=d,2=\t]{\polytope{\Pi}_{#2}^{#1}}
\newcommandx{\PP}{\polytope{\Pi}}
\newcommand\cyc{\polytope{Cyc}}

\newcommandx{\HypSimpl}[2][1=n,2=k]{\polytope{\Delta}(#1,#2)}

\newcommandx{\MPP}[2][1=\polytopeP,2=\b c]{\polytope{M}_{#2}(#1)}
\newcommandx{\MPPHypSimpl}[2][1=n,2=k]{\polytope{M}(#1,#2)}

\newcommand{\DCe}[1][\polytopeP]{\mathbb{DC}_e(#1)}
\newcommand{\DC}[1][\polytopeP]{\mathbb{DC}(#1)}
\newcommand{\DCZG}[1][G]{\mathbb{DC}(\polytope{Z}_{#1})}

\newcommandx{\ZG}[1][1=G]{\mathsf{Z}_{#1}}

\newcommandx{\Asso}[2][1=n,2={}]{\mathsf{Asso}^{#2}(#1)} 
\newcommandx{\dZono}[1][1=\b{h}]{\mathsf{D}_{#1}} 
\newcommand{\simplex}{\polytope{\Delta}} 

\newcommandx{\Fan}[1][1=F]{\mathcal{#1}} 
\newcommandx{\nestedFan}[1][1=\quiver]{\mathcal{F}(#1)} 

\newcommandx{\ray}[1][1=r]{\b{#1}} 
\newcommandx{\rays}[1][1=R]{\b{#1}} 
\newcommandx{\Perm}[1][1=n]{\polytope{Perm}_{#1}}

\newcommandx{\gArr}[1][1=G]{\mathcal{A}_{#1}} 
\newcommandx{\gFan}[1][1=G]{\Fan_{#1}} 
\newcommandx{\gFanO}[1][1=G]{\widehat{\Fan}_{#1}} 
\newcommandx{\cc}[1][1=G]{\mathbb{K}_{#1}} 

\newcommandx{\braid}[1][1=n]{\mathcal{B}_{#1}} 
\newcommandx{\sbraid}[1][1=n]{\widehat{\mathcal{B}}_{#1}} 

\newcommandx{\coefficient}[3][1={\b{s}}, 2=\b{r}, 3=\b{r}']{\alpha_{#2,#3}(#1)} 
\newcommandx{\virtualPolytopes}[1][1=d]{\mathbb{V}^{#1}} 
\newcommandx{\VDP}[1][1=n]{\mathbb{VDP}^{#1}} 
\newcommandx{\CVDP}[1][1=n]{\overrightarrow{\mathbb{VDP}}^{#1}} 
\newcommand{\VD}[1][1=n]{\mathbb{VD}} 
\newcommandx{\opcone}[1][1={\mu,\omega}]{\polytope{C}_{#1}}
\newcommandx{\orcone}[1][1={\omega}]{\polytope{C}_{#1}}

\newcommandx{\graphG}[1][1=G]{#1} 
\newcommandx{\hypergraph}[1][1=H]{\graphG[#1]} 
\newcommandx{\tube}[1][1=t]{\mathsf{#1}} 
\newcommandx{\tubes}[1][1=\graphG]{\building#1} 
\newcommandx{\tubing}[1][1=T]{\mathsf{#1}} 

\newcommand{\building}{\mathcal{B}} 
\newcommandx{\nested}[1][1=N]{\mathcal{#1}} 

\newcommandx{\enhancedStep}[3][1=i, 2=j, 3=a]{#1 \xrightarrow{#3} #2}
\newcommandx{\step}[2][1=i, 2=j]{#1 \rightarrow #2}
\newcommandx{\enhancedStepx}[3][1=x, 2=y, 3=z]{#1 \xrightarrow{#3} #2}
\newcommandx{\enhancedStepZ}[3][1=x, 2=y, 3=Z]{#1 \xrightarrow{#3} #2}

\renewcommand\t{\b t}%

\newcommandx{\ConstrainedMALIntrinsic}[2][1=n,2=d]{\go M_{#1,#2}}

\newcommandx{\AssGraph}[2][1=n,2=3]{\mathit{Asso}_{#1}^{#2}}

\newcommandx{\Fpolytope}[3][1=d,2=A,3=\b t]{\polytopeP_{#1}^f\left(#2,#3 \right)}
\newcommandx{\Bpolytope}[3][1=d,2=A,3=\b t]{\polytopeP_{#1}^b \left(#2,#3 \right)}

\newcommandx{\FibPol}[3][1=\polytopeP,2=\polytopeQ,3=\pi]{\polytope{\Sigma}_{#3}(#1,#2)}

\newcommandx{\FibPolCyc}[2][1=d,2=\t]{\polytope{\Sigma}^{#1}_{2}(#2)}
\newcommandx{\PolProj}[3][1=\polytopeP,2=\polytopeQ,3=\pi]{#3~:~#1\to #2}

\newcommandx{\HOmega}[3][1=\kappa,2=\t,3=d]{\Omega_{\kappa}^d(\t)}
\newcommandx{\rHOmega}[3][1=\kappa,2=\t,3=d]{\overline{\Omega}_{\kappa}^d(\t)}
\newcommandx{\Ppolytope}[3][1=d,2=T,3=\t]{\polytopeQ^+_{#1}(#2,#3)}
\newcommandx{\Npolytope}[3][1=d,2=T,3=\t]{\polytopeQ^-_{#1}(#2,#3)}

\newcommandx{\gVdM}[3][1=n,2=k,3=\b \lambda]{\text{VdM}_{#1,#2}(#3)}
\newcommandx{\VdM}[2][1=n,2=\b\lambda]{\text{VdM}_{#1}(#2)}

\newcommand{\inner}[1]{\left<#1\right>}

\usepackage{todonotes}

\AtEndDocument{%
  \par
  \medskip
  \begin{tabular}{@{}l@{}}%
    Germain Poullot, \textsc{Universit\"at Osnabr\"uck}\\
    \textit{E-mail}: \texttt{germain.poullot@uni-osnabrueck.de}
  \end{tabular}}

\title{Deformation cones of graphical zonotopes for $K_4$-free graphs}
\author{Germain POULLOT}
\date{}

\begin{document}

\maketitle

\begin{abstract}
In this paper, we compute a triangulation of certain faces of the submodular cone.
More precisely, graphical zonotopes are generalized permutahedra, and hence their deformation cones are faces of the submodular cone.
We give a triangulation of these faces for graphs without induced complete sub-graph on 4 vertices.
We deduce the rays of these faces: Minkowski indecomposable deformations of these graphical zonotopes are segments and triangles.

Besides, computer experiments lead to examples of graphs without induced complete sub-graph on 5 vertices, whose graphical zonotopes have high dimensional Minkowski indecomposable deformations.
\end{abstract}

In this paper, $\R^d$ is the Euclidean space of dimension $d$, endowed with its scalar product $\inner{\cdot,\cdot}$ and canonical basis $(\b e_1, \dots, \b e_d)$.
For $X\subseteq[d]$, we have $\b e_X = \sum_{i\in X} \b e_i$, and $\simplex_X = \conv\bigl(\b e_i ~;~ i\in X\bigr)$.
For a polytope $\polytopeP$ and $\lambda\in \R$, we denote $\lambda\polytopeP = \{\lambda\b x ~;~ \b x \in \polytopeP\}$, especially when $\lambda = -1$.

\vspace{-0.25cm}

\section*{Acknowledgments}
The author wants to deeply thank two persons.
First, Vic Reiner, during the defense of PhD thesis of the author, asked if something could be said of the dimensions of the rays of the deformation cone of a graphical zonotope when the graph has triangles.
The main theorem of this paper answers this question.
Besides, Vic Reiner asked the author to compute \Cref{exmp:counter-example1} whose result is the center of \Cref{sec:High_dimensional_summands}.
Second, Arnau Padrol, took time in Oberwolfach to listen to the speculations of the author, and proposed to focus on edge-length deformations instead of height deformations.
This change of perspective was the missing keystone for assembling all the ideas together.

Thanks also go to Martina Juhnke for her re-reading, to the MF Oberwolfach for its benevolent atmosphere, and to SageMath without which no conjecture would have been made in the first place.

Computer experiments were done with Sage 9.1 \cite{Sage}, building my own code for deformation cones.
Since, in Sage 10.6, a direct method was introduced to compute deformation cones.

\vspace{-0.25cm}

\tableofcontents


\input{1_Introduction}


\input{2_K4free}


\input{3_WithK4}


\input{4_Open_Questions}


\bibliographystyle{alpha}
\bibliography{Biblio}
\label{sec:biblio}

\end{document}

%% file: 1_Introduction.tex
\section{Introduction}

Originally introduced by Edmonds in 1970 under the name of \defn{polymatroids} as a polyhedral generalization of matroids in the context of linear optimization~\cite{Edmonds}, the \defn{generalized permutahedra} were rediscovered by Postnikov in 2009~\cite{Postnikov2009}, who initiated the investigation of their rich combinatorial structure. They have since become a widely studied family of polytopes that appears naturally in several areas of mathematics, such as algebraic combinatorics~\cite{PostnikovReinerWilliams,ArdilaBenedettiDoker,AguiarArdila}, optimization~\cite{SubmodularFunctionsOptimization},  game theory~\cite{DanilovKoshevoy2000}, statistics~\cite{MortonPachterShiuSturmfelsWienand2009,MohammadiUhlerWangYu2018}, and economic theory~\cite{JoswigKlimmSpitz2022}. The set of deformed permutahedra can be parameterized by the cone of \defn{submodular functions}~\cite{Edmonds,Postnikov2009}, and is hence called the \defn{submodular cone}.

The search for irredundant facet descriptions of deformation cones of particular families of combinatorial polytopes has received considerable attention recently, leading to powerful results such as~\cite{CastilloLiu2020, ACEP-DeformationsCoxeterPermutahedra,  AlbertinPilaudRitter, CastilloDoolittleGoecknerRossYing-MinkowskiSummandCube, PadrolPaluPilaudPlamondon, PPP2023Nesto, PPP2023gZono, BazierMatteDouvilleMousavandThomasYildirim}.
One of the motivations sparking this interest arises from the \defn{amplituhedron program} to study scattering amplitudes in mathematical physics~\cite{ArkaniHamedTrnka-Amplituhedron}. As described in \cite[Sec.~1.4]{PadrolPaluPilaudPlamondon}, the deformation cone provides canonical realizations of a polytope (seen as a \defn{positive geometry}~\cite{ArkaniHamedBaiLam-PositiveGeometries}) in the positive region of the kinematic space, akin to those of the associahedron in~\cite{ArkaniHamedBaiHeYan}.

However, faces of the submodular cone are far from being well understood: determining its rays for instance remains an open problem since the 1970s~\cite{Edmonds}.
In this paper, we will describe a triangulation of certain faces of the submodular cone, and deduce their rays and faces of dimension~2.
These faces are in one-to-one correspondence with graphical zonotopes for graphs with no induced complete graph on 4 vertices, see \Cref{sec:K4_free_graphs}:

\begin{thmUniv}[\Cref{thm:DCZG_K4-free_graphs,cor:DCZG_graph_is_bi-pyramid_graph}]
Let $G$ be a $K_4$-free graph, $E$ its set of arcs and $T$ its set of triangles.
The deformation cone $\DCZG$ of the graphical zonotope $\ZG = \sum_{\{i, j\}\in E} [\b e_i, \b e_j]$ has dimension $|E| + |T|$, and is triangulated by the $2^{|T|}$ simplicial cones defined in \Cref{def:SimplicialConesForTriangulation}.

The cone $\DCZG$ has $|E| + 2|T|$ rays, associated to the segments $\simplex_e$ for $e\in E$, and the triangles $\simplex_t$ and $-\simplex_t$ for $t\in T$.
Moreover, the cone $\DCZG$ has $\binom{|E|+2|T|}{2} - |T|$ faces of dimension 2.
\end{thmUniv}

Furthermore, in \Cref{sec:High_dimensional_summands}, we show that such a theorem does not hold for graphs containing a complete graph on 4 vertices (even for some $K_5$-free graphs), see \Cref{exmp:counter-example1}.

\section{Preliminaries}
Before being able to state our main results, we will need to introduce deformations (\Cref{ssec:Deformations_intro}) and graphical zonotopes (\Cref{ssec:ZG_intro}), and to address the cases of 2-dimensional graphical zonotopes (\Cref{ssec:2d_ZG}) and graphical zonotopes for triangle-free graphs (\Cref{ssec:Triangle_free}).

\subsection{Deformation cone and edge deformations}\label{ssec:Deformations_intro}

There are a lot of ways to define deformations of polytopes.
The aim is to study how one can write a polytope as a \defn{Minkowski sum} of other polytopes: $\polytopeQ  + \polytope{R} := \{\b q + \b r ~;~ \b q\in \polytopeQ,~ \b r\in \polytope{R}\}$.
A deformation of a polytope $\polytopeP$ can be defined as a \defn{weak Minkowski summand} of $\polytopeP$, that is to say a polytope $\polytopeQ$ such that there exists $\lambda > 0$ and a polytope $\polytope{R}$ with $\lambda \polytopeP = \polytopeQ + \polytope{R}$.
Though theoretically important, this definition is unhandy and we will prefer a more geometric perspective.
For a face $\polytopeF$ of a polytope $\polytopeP$, we denote by $\c N_\polytopeP(\polytopeF)$ the outer normal cone of $\polytopeF$, \ie $\c N_\polytopeP(\polytopeF) = \{\b c\in \R^d ~;~ \polytopeP^{\b c} = \polytopeF\}$ where $\polytopeP^{\b c} = \{\b x\in \polytopeP ~;~ \inner{\b x, \b c} = \max_{\b y\in \polytopeP} \inner{\b y, \b c}\}$.
The \defn{normal fan} of $\polytopeP$ is then the collection of its normal cones: $\c N_\polytopeP := \bigl(\c N_\polytopeP(\polytopeF) ~;~ \polytopeF \text{ face of } \polytopeP\bigr)$.
The normal fan contains combinatorial and geometric information.
In particular, if $\c N_\polytopeP = \c N_\polytopeQ$, then $\polytopeP$ and $\polytopeQ$ are combinatorially isomorphic, but the converse is false.
A fan $\c F$ \defn{coarsens} a fan $\c G$, denoted $\c F\trianglelefteq \c G$, if all the cones of $\c G$ are contained in cones of $\c F$.
As $\c N_{\polytopeQ + \polytope{R}}$ is the common refinement of $\c N_\polytopeQ$ and $\c N_{\polytope{R}}$, if $\polytopeQ$ is a weak Minkowski summand of $\polytopeP$, then $\c N_\polytopeQ$ coarsens $\c N_\polytopeP$.
The converse can also be shown, motivating to the following definition.

\begin{definition}
A \defn{deformation} (or \defn{weak Minkowski summand}) of a polytope $\polytopeP$ is a polytope $\polytopeQ$ such that $\c N_\polytopeQ$ coarsens $\c N_\polytopeP$.
The \defn{deformation cone} of $\polytopeP$ is:
$$\DC := \bigl\{\polytopeQ ~;~ \c N_\polytopeQ \trianglelefteq \c N_\polytopeP\bigr\}$$
\end{definition}

Note that $\DC$ is indeed a cone because $\c N_{\lambda \polytopeQ} = \c N_\polytopeQ$ for and $\lambda > 0$, and $\c N_{\polytopeQ + \polytope{R}}$ is the common refinement of $\c N_\polytopeQ$ and $\c N_{\polytope{R}}$.
Moreover, if $\polytopeQ$ is a deformation of $\polytopeP$, then so is any translation of $\polytopeQ$: to ease the presentation, we quotient $\DC$ by translation, and keep only one representative per class of translation, \ie we consider $\DC$ without its lineality space.

In order to understand this cone more in more details, we parameterize it.
For the purpose of this paper, we choose to parameterize it as the edge-length deformation cone $\DCe$, presented thereafter.
In \cite[Appendix 15]{PostnikovReinerWilliams}, the authors prove that $\DC$ (without its lineality space) is linearly isomorphic to $\DCe$, and state the following \Cref{thm:polygonal_face_equation}.
We give here a condensed presentation of the main ideas, referring the reader to \cite[Appendix 15]{PostnikovReinerWilliams} for the proofs.

Denoting by $E(\polytopeP)$ the set of edges of a polytope $\polytopeP$, we associate to $\polytopeP$ its \defn{edge-length vector $\b\ell(\polytopeP)\in\R_+^{E(\polytopeP)}$} whose coordinate $\ell(\polytopeP)_e$ is the length of the edge $e\in E(\polytopeP)$.
Conversely, we want to associate a polytope $\polytopeQ_{\b\ell}$ to a (coordinate-wise) positive vector $\b\ell\in\R_+^{E(\polytopeP)}$.
For each edge $e\in E(\polytopeP)$, pick a unit vector $\b u_e$ by arbitrarily choosing an orientation of the edges.
Consider the (un-directed) graph $\Gamma_\polytopeP$ of $\polytopeP$, and fix a vertex $\b v$ of $\polytopeP$, then construct the following polytope \linebreak $\polytopeQ_{\b \ell} = \conv\bigl\{\sum_{e\in\c P} \varepsilon_e^{\c P} \ell_e \b u_e ~~;~ \c P \text{ directed edge-path in } \Gamma_\polytopeP \text{ starting at } \b v\bigr\}$ where $\varepsilon_e^{\c P} = 1$ if the direction of $e$ in $\c P$ agrees with the one chosen for $\b u_e$, and $\varepsilon_e^{\c P} = -1$ else way.
Choosing another vertex $\b v$ of $\polytopeP$ amounts to translating $\polytopeQ_{\b\ell}$, giving rise to a normally equivalent polytope.
The \defn{edge-length deformation cone} of $\polytopeP$ is:
$$\DCe = \Bigl\{\b\ell\in \R_+^{E(\polytopeP)} ~~;~~  \c N_{\polytopeQ_{\b \ell}} \trianglelefteq \c N_\polytopeP\Bigr\}$$

There is an efficient way to check if some positive vector $\b\ell\in\R_+^{E(\polytopeP)}$ gives rise to a deformation of~$\polytopeP$.
Indeed, let $\c F_k(\polytopeP)$ be the set of $k$-dimensional faces of $\polytopeP$.
For each 2-dimensional face $\polytopeF\in\c F_2(\polytopeP)$ and $e\in E(\polytopeF)$, let $\b n_e^\polytopeF$ the outer normal unit vector of $e$ in the plane $\text{aff}(\polytopeF)$.
Then the \defn{polygonal face equation} given by $\polytopeF$ is: $\sum_{e\in E(\polytopeF)} \ell_e \b n^\polytopeF_e = \b 0$.

\begin{theorem}\label{thm:polygonal_face_equation}\emph{(\cite[Theorem 15.3]{PostnikovReinerWilliams}).}
For a polytope $\polytopeP$, the edge-length deformation cone is the intersection of $\R_+^{E(\polytopeP)}$ with (the kernel of) the polygonal face equations:
$$\DCe = \R_+^{E(\polytopeP)} \cap \bigcap_{\polytopeF\in \c F_2(\polytopeP)} \bigl\{\b\ell ~;~ \sum_{e\in E(\polytopeF)} \ell_e \b n^{\polytopeF}_e = \b 0\bigr\}$$

Moreover, $\DCe$ is linearly isomorphic to $\DC$ without its lineality space.
\end{theorem}

\begin{remark}
This description makes it clear that the (edge-length) deformation cone is a polyhedral cone, but has the drawback that the deformation cone is embedded in very high dimension.
Computer implementations more often uses other parameterizations of the deformation cone, such as the height deformation cone (see \cite{ChapotonFominZelevinsky,GelfandKapranovZelevinsky,PilaudSantos-quotientopes,PadrolPaluPilaudPlamondon,PPP2023Nesto,PPP2023gZono}).
\end{remark}

Last but not least, by construction, the face lattice of $\DC$ is the lattice of classes of normally equivalent deformations of $\polytopeP$.
In particular, the following allows us to study the faces of $\DC$:

\begin{theorem}\label{thm:faces_of_DC_are_DC}
If $\polytopeQ$ is a deformation of $\polytopeP$, then $\DC[\polytopeQ]$ is a face of $\DC$.    
\end{theorem}

\begin{proof}
This is mainly equivalent to \cite[Theorem 7]{McMullen-typeCone}, but we give here a self contained proof.

It follows from the edge-length deformation cone description.
As $\polytopeQ = \polytopeQ_{\b\ell}$ for some $\b \ell\in \DCe$, consider the face $\polytopeC$ of $\DCe$ such that $\b\ell$ lies in the interior of $\polytopeC$.
Then all $\b\ell'\in\polytopeC$ respect the polygonal face equations of $\polytopeQ_{\b\ell}$, and thus $\b\ell'\in \DCe[\polytopeQ]$.
Conversely, any deformation of $\polytopeQ$ can be written as $\polytopeQ_{\b \ell'}$ for some $\b \ell'\in\polytopeC$ because $\polytopeC$ respects the polygonal face equations of $\polytopeQ$.
\end{proof}

\begin{remark}
One can define the polygonal face equations in a dual fashion:
for each 2-dimensional face $\polytopeF\in\c F_2(\polytopeP)$ and each $e\in E(\polytopeF)$, fix a unit vector $\b u_e^\polytopeF$ parallel to the direction of $e$ (by choosing an orientation) such that, endowed with this orientation, $\Gamma_\polytopeF$ is a directed cycle.
Then the (dual) polygonal face equation given by $\polytopeF$ is: $\sum_{e\in E(\polytopeF)} \ell_e \b u^\polytopeF_e = \b 0$.

It is not tedious to see that this definition and the above lead to the same system of equations.
However, the outer-normal setting generalizes more easily to higher dimensions:
for a $k$-dimensional face $\polytopeF\in\c F_k(\polytopeP)$, consider $\polytopeF$ as a full-dimensional polytope embedded in its affine hull $\text{aff}(\polytopeF)$, and let $\b n^\polytopeF_{\polytope{G}}$ be the outer normal (unit) vector to the facet ${\polytope{G}}$ of $\polytopeF$.
Then the $k$-face equation given by $\polytopeF$ is: $\sum_{\polytope{G}\text{ facet of } \polytopeF} \ell_{\polytope{G}} \b n^\polytopeF_{\polytope{G}} = \b 0$.
A vector $\b\ell\in \R^{\c F_{k}(\polytopeP)}$ is a Minkowski $k$-weight \cite{McMullen1996} if it respects all $(k+1)$-face equations: deformations are equivalent to positive Minkowski 1-weights.
The reader shall consult \cite{McMullen1996}, especially its Lemma 8.1 for further details on the construction.
\end{remark}

\subsection{Graphical zonotopes}\label{ssec:ZG_intro}

In order to differentiate between graphs an polytopes, we will say that our graphs consists of \defn{nodes} linked together by \defn{arcs}.
All our graphs will have node set $[n] := \{1, \dots, n\}$, if not mentioned otherwise.
For a graph $G$, we denote by $V(G)$ its node set and $E(G)$ its arc set, and abbreviate it by $V$ and $E$ when the context is clear.
Recall that for $X\subseteq[n]$, we have $\b e_X = \sum_{i\in X} \b e_i$.

\begin{definition}
The \defn{graphical zonotope} associated to a graph $G = (V, E)$ is the polytope $\ZG \subset \R^V$ defined as the following Minkowski sum, where $(\b e_i)_{i \in V}$ is the standard basis of $\R^V$:
$$\ZG = \sum_{e\in E} \simplex_e = \sum_{\{i, j\} \in E} [\b e_i, \b e_j]$$
\end{definition}

We need to understand both the combinatorics and the geometry of graphical zonotopes.
We will attempt to give an in depth presentation of graphical zonotopes, the reader can refer to \cite{Stanley2007,OrientedMatroids} for a general presentation of the subject, or to \cite{PPP2023gZono} for another presentation of the deformation cone of graphical zonotopes.

Denoting the collection of connected components of $G$ by \defn{$\b{cc}(G)$}, remark that for any \linebreak $C\in \b{cc}(G)$, we have $\inner{\b x,\b e_C} = |E(C)|$ for all $\b x\in \ZG$.
Hence, the polytope $\ZG$ lives in the intersection of affine hyperplanes $\bigcap_{C\in \b{cc}(G)} \bigl\{\b x\in\R^V ~;~ \inner{\b x,\b e_C} = |E(C)|\bigr\}$.
As $\ZG$ is full dimensional in this space, it follows that: $\dim \ZG = |V(G)| - |\b{cc}(G)|$.

Moreover, if $G$ has two connected components $C_1, C_2 \in \b{cc}(G)$, with $i\in C_1$, $j\in C_2$, then $\ZG$ is normally equivalent to $\ZG[H]$ where $H$ is obtained from $G$ by identifying any node $i\in C_1$ with any node $j\in C_2$.
Indeed, consider the projection $\pi_{i, j} : \R^V \to \R^{V\ssm \{j\}}$ defined by $\pi_{i, j}(\b e_j) = \b e_i$ and $\pi_{i, j}(\b e_k) = \b e_k$ for $k\ne j$.
A quick scribble gives that $\pi_{i, j}(\ZG) = \ZG[H]$.
Therefore, to deal with a non-connected graph, we can always make it connected by ``gluing" its connected components (\ie taking a 1-sum of its connected components).

\vspace{0.2cm}

Given a collection $\mu$ of disjoint subsets of $V$, the \defn{restriction} of $G$ to $\mu$ (or the \defn{sub-graph of $G$ induced on $\mu$}), denoted $\restrG$ is the sub-graph of $G$ whose node set is $\bigsqcup_{X\in\mu} X$ and whose arcs are the arcs $\{i, j\} \in E$ with $i, j \in X$ for some $X\in \mu$.
Besides, the \defn{contraction} of $\mu = (M_1, \dots, M_r)$ in $G$, denoted $\contrG$, is the graph whose nodes are $M_1, \dots, M_r$ and in which there is an arc between $M_a$ and $M_b$ if there exists $i\in M_a$ and $j\in M_b$ with $\{i, j\} \in E$.

A collection of disjoint subsets $\mu$ is a \defn{partition} of $V$ if $V = \bigsqcup_{X\in \mu} X$.
When necessary, we identify any collection of disjoint subsets $\mu$ of $V$ with the partition obtained by completing it with singletons: $\underline{\mu} := \mu \cup \bigl(\{i\} ~;~ i \notin \bigsqcup_{X\in\mu} X\bigr)$.
By convention, if $\mu = (X)$ consists in only one set $X\subseteq V$, then we will write $\restrG[X]$ instead of $\restrG[(X)]$, and $\contrG[X]$ instead of $\contrG[(\underline{X})]$.

\vspace{0.2cm}

An \defn{acyclic orientation} of a graph $G$ is an orientation of each arc of $G$ such that the induced directed graph has no directed cycle.
We denote by \defn{$\c O(G)$} the set of all acyclic orientations of $G$.

An \defn{ordered partition} of a graph $G = (V, E)$ is a pair $(\mu, \rho)$ where $\mu$ is a partition of $V$ and $\rho \in \c O(\contrG)$.
For $u, v\in V$, we write \defn{$u \xrightarrow{\rho} v$} if there exists a directed path, respecting the orientation $\rho$, from the part of $\mu$ containing $u$ to the one containing $v$.
An ordered partition $(\mu, \rho)$ \defn{refines} $(\mu', \rho')$ when every part of $\mu$ is contained in a part of $\mu'$, and $\rho'$ is the orientation on $\contrG[\mu']$ induced by $\rho$.
This refinement order defined a poset structure on the set of ordered partitions.

\begin{proposition}\label{prop:Normal_fan_ZG}\emph{(\cite[Definition 2.5]{Stanley2007}).}
For a graph $G = (V, E)$, the normal fan of the graphical zonotope $\ZG$ is the fan induced by the \defn{graphical hyperplane arrangement} $\c H_G := \bigl(H_{i, j} ~;~ \{i, j\} \in E\bigr)$ with $H_{i, j} := \bigl\{\b x\in \R^V ~;~ x_i = x_j\bigr\}$.

Consequently, the lattice of faces of $\ZG$ is (anti-)isomorphic to the lattice of ordered partitions of $G$.
More precisely, the face associated to the ordered partition $(\mu, \rho)$ is normally equivalent to $\ZG[\restrG]$ and its normal cone is:
$$\polytopeC_{\mu, \rho} := \left\{\b x ~;~ \begin{array}{lcll}
     x_u & = & x_v & \text{if } u, v \text{ are in the same part of } \mu \\
     x_u & \leq & x_v & \text{if } u \xrightarrow{\rho} v
\end{array} \right\}$$
\end{proposition}

\begin{remark}\label{rmk:knowing_if_deformation}
Note that in order to know that $\polytopeP$ is a deformation of $\ZG$, one needs to verify that $\c N_\polytopeP$ coarsens $\c N_{\ZG}$.
This is equivalent to checking that all co-dimension 1 cones of $\c N_\polytopeP$ are contained in a hyperplane $H_{i, j}$ for some $\{i, j\} \in E$, because $\c N_{\ZG}$ is the fan induced by the \defn{graphical arrangement} $\c H_G := \bigl(H_{i, j} ~;~ \text{for } \{i, j\}\in E\bigr)$.
Hence, $\polytopeP$ is a deformation of $\ZG$ if and only if all the edges of $\polytopeP$ are dilates of $\b e_i - \b e_j$ for some $\{i, j\} \in E$.
When $G = K_n$ is the complete graph, we recover the well-known fact that $\polytopeP$ is a generalized permutahedron if and only if all its edges are dilates of of $\b e_i - \b e_j$ for some $1\leq i < j\leq n$.
\end{remark}

According to \Cref{prop:Normal_fan_ZG}, we can have a full understanding of $k$-dimensional faces of $\ZG$ for a graph $G$ by studying $\ZG[H]$ for induced sub-graphs $H$ of $G$ with $\dim\ZG[H] = |V(H)| - |\b{cc}(H)| = k$.
Especially, we will need a precise understanding of the faces of dimension 3 and less of $\ZG$:
\begin{compactenum}
\item[Dim 0:] Vertices of $\ZG$ are in bijection with acyclic orientations of $G$, \ie $V(\ZG) \simeq \c O(G)$.
\item[Dim 1:] Edges of $\ZG$ are in bijection with pairs $(e, \rho)$ where $e\in E$ is an arc of $G$ and $\rho\in \c O(\contrG[e])$.
\end{compactenum}

The possible faces of dimension $2$ and $3$ are described in the following \Cref{fig:table_small_dim_ZG}.
For each induced sub-graph $H$ of $G$ of the form indicated in the second column of \Cref{fig:table_small_dim_ZG}, there are $\left|\c O(\contrG[\b {cc}(H)])\right|$ many faces of $\ZG$ isomorphic to $\ZG[H]$.
Recall that graphical zonotopes of non-connected graphs are normally equivalent to graphical zonotopes of any 1-sum of their connected components.

Two acyclic orientations $\rho, \rho'\in\c O(G)$ differ on the orientation of only one arc of $G$ if and only the two corresponding vertices of $\ZG$ share a common edge.
In this case, we will say that $\rho$ into $\rho'$ are linked by a \defn{flip}.
The \defn{flip graph} is the graph whose node set is $\c O(G)$, and edges are between orientations linked by a flip.
It is isomorphic to the graph of $\ZG$, and hence connected.

\input{Figures/Table_small_dimensional_ZG}


Deformation cones of graphical zonotopes have been studied before in \cite{PPP2023gZono}.
There, the authors gave a facet description of $\DCZG$ for all graphs $G$.
We will not use such a precise description in what follows.
The present work focuses on the rays of $\DCZG$, but we recall their result, as we will need the dimension and the number of facets of $\DCZG$.
The dimension was already computed by Raman Sanyal and Josephine Yu (personal communication), who computed the
space of Minkowski 1-weights of graphical zonotopes.

\begin{theorem}\label{thm:DCZG_dim_and_nbr_facet}\emph{(\cite[Corollary 2.7]{PPP2023gZono}).}
Let $G = (V, E)$ be a graph, $T$ its set of triangle, and $\Omega(G)$ the number of its (induced) cliques\footnote{Singletons are not considered cliques: a clique is any $X\subseteq V$ with $\restrG[X] \simeq K_r$ for $r = |X| \geq 2$.}, then $\dim \DCZG = \Omega(G)$, and $\DCZG$ has $\sum_{e\in E}2^{\left|\left\{t\in T ~;~ e\,\subseteq\, t\right\}\right|}$ facets.
\end{theorem}

\begin{remark}
According to this \Cref{thm:DCZG_dim_and_nbr_facet}, the dimension of $\DCZG$ is $\Omega(G)$, so one can expect to easily find $\Omega(G)$ rays of $\DCZG$.
Indeed, for each induced clique on node set $K$, the simplex $\simplex_K$ is Minkowski indecomposable (all its 2-faces are triangles, so the polygonal face equation forces all lengths to be equal), and it is a deformation of $\ZG$ (by \Cref{rmk:knowing_if_deformation}), so is it associated to a ray of $\DCZG$.
Furthermore, $-\simplex_K$ is also associated to a ray of $\DCZG$, giving $2\Omega(G) - |E|$ many rays (note that $-\simplex_K$ and $\simplex_K$ are normally equivalent if and only if $|K| = 2$).
\end{remark}

\subsection{Deformations of 2-dimensional graphical zonotopes}\label{ssec:2d_ZG}\label{ssec:Parallelograms_and_hexagons}

On the one hand, faces of dimension 2 of a polytope governs its deformations through the polygonal face equations.
On the other hand, faces of dimension 2 of graphical zonotopes correspond to 2-dimensional graphical zonotopes, \ie parallelograms and regular hexagons.
This short subsection describes the polygonal face equations and deformations of these two polygons.

\paragraph{Parallelograms}

When $G$ is a graph with $2$ arcs $e, f$ (disjoint or not), then its graphical zonotope $\ZG$ is a parallelogram.
Its 4 vertices correspond to the acyclic orientations of these arcs, and each of its 4 edges correspond to the contraction of one arc and an orientation of the other.

As $\ZG$ is a parallelogram, its opposite edges are parallel, and hence the polygonal face equation boils down to the opposite edges having the same length, see \Cref{fig:P3_decorated_parallelogram}.
Labeling $(e, \overrightarrow{f})$, $(e, \overleftarrow{f})$, $(f, \overrightarrow{e})$, and $(f, \overleftarrow{e})$ the edges of $\ZG$, and labeling $\b\ell\in \R_+^{E(\ZG)}$ accordingly, we have:
$$\ell_{e, \overrightarrow{f}} ~=~ \ell_{e, \overleftarrow{f}} ~~~~~~~~\text{ and }~~~~~~~~ \ell_{f, \overrightarrow{e}} ~=~ \ell_{f, \overleftarrow{e}}$$

\begin{figure}
    \centering
    \begin{subfigure}[b]{0.49\textwidth}
    \centering
    \includegraphics[width=0.6\linewidth]{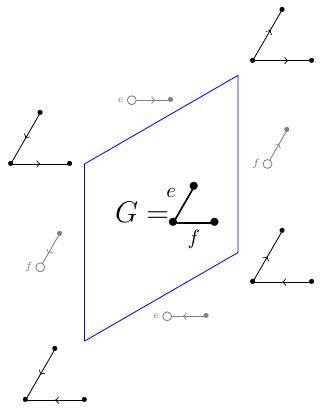}
    \caption{Graphical zonotope for the path on 3 nodes}
    \label{fig:P3_decorated_parallelogram}
    \end{subfigure}
    \begin{subfigure}[b]{0.49\textwidth}
    \centering
    \includegraphics[width=0.7\linewidth]{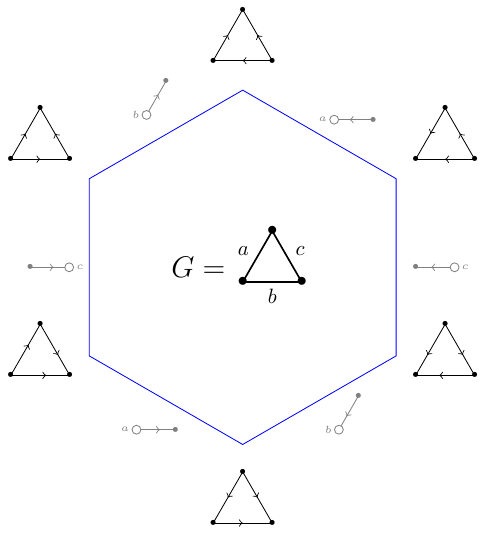}
    \caption{Graphical zonotope for $K_3$}
    \label{fig:K3_decorated_hexagon}
    \end{subfigure}
    \caption{Each vertex and edge is labeled by the corresponding ordered partition: acyclic orientations of $G$ for vertices, and acyclic orientations of a contraction of $\contrG[x]$ for a pair of parallel edges.}
\end{figure}

The (edge-length) deformation cone of $\ZG$ is the intersection of $\R_+^{E(\ZG)} \simeq \R_+^4$ with the (kernel of) these two equations.
We have $|E| = 2$, $|T| = 0$ and $\Omega(G) = 2$, so, by \Cref{thm:DCZG_dim_and_nbr_facet}, this gives rise to a 2-dimensional (simplicial) cone.
Its 2 rays are associated to the segments $\simplex_e$ and $\simplex_f$.
Any deformation $\polytopeP$ of $\ZG$ can thus be written (in a unique way) as $\polytopeP = \lambda_e \simplex_e + \lambda_f \simplex_f$ for some $\lambda_e, \lambda_f \geq 0$.
There are 4 normal equivalence classes of deformations of $\ZG$: the class of $\ZG$, the class of $\simplex_e$, the class of $\simplex_f$, and the class of $\b 0$ (the $0$-dimensional polytope).

\paragraph{Regular hexagons}

When $G = K_3$ is the complete graph on $3$ nodes, \ie a triangle, then $\ZG[K_3]$ corresponds to the 2-dimensional permutahedron: a regular hexagon.
This case is a bit more convoluted than the parallelogram, but remains manageable.

Its 6 vertices correspond to the 6 acyclic orientations of the triangle (only 2 orientations give rise to a cycle).
Its 6 edges come in 3 pairs of parallel edges.
Each pair of parallel edges corresponds to the contraction of one arc of $G$.
We name $a$, $b$ and $c$ the arcs of $G$.
The graph $\contrG[a]$ is just an arc, each of its 2 orientations giving rise to an edge of $\ZG$.
We label the edges of $\ZG$ by $(a, \to)$, $(a, \leftarrow)$, respectively $(b, \to)$, $(b, \leftarrow)$, and $(c, \to)$, $(c, \leftarrow)$, where $\to$ is the orientation of $\contrG[a]$, respectively $\contrG[b]$ and $\contrG[c]$, from the contracted arc towards the other node, and conversely for $\leftarrow$.

It is not hard to check that the polygonal face equation boils down to the two equations:
\begin{equation}\label{eqn:hexagonal_face_eqn}
\ell_{a, \to} - \ell_{a, \leftarrow} ~~=~~ \ell_{b, \to} - \ell_{b, \leftarrow} ~~=~~ \ell_{c, \to} - \ell_{c, \leftarrow}    
\end{equation}


This motivates the following crucial definition.

\begin{definition}\label{def:step_hexagon}
For any hexagonal face $\polytopeH$ of a deformation of a graphical zonotope, the \defn{step of $\polytopeH$}, denoted \defn{$\delta_\polytopeH$}, is the absolute value of the difference of the lengths of opposite edges.
\end{definition}

\begin{remark}
This definition can also be made for any hexagonal face of any generalized permutahedron.
Note also that parallelogram faces can be interpreted as degenerated hexagonal faces: not only the difference of length of opposite edges is the same, but this step is actually $0$.
However, not all quadrilaterals in generalized permutahedra are parallelograms.
In \Cref{fig:DCZG_K3} (bottom right), there is a quadrilateral in which only one pair opposite edges are parallel.
\end{remark}

Finally, the deformation cone of $\ZG[K_3]$ is the intersection of $\R_+^{E(\ZG[K_3])} \simeq \R_+^6$ with (the kernel) of \Cref{eqn:hexagonal_face_eqn}.
We have $|E| = 3$, $|T| = 1$ and $\Omega(G) = 4$, so, by \Cref{thm:DCZG_dim_and_nbr_facet}, we know that $\DCZG$ is a 4-dimensional cone with 6 facets.
Here, the explicit computation is not as easy as for the case of the parallelograms, see \cite[Example 2.8]{PPP2023gZono} for the details.
We obtain a 4-dimensional cone with 5 rays: $\simplex_a$, $\simplex_b$, $\simplex_c$ which are segments, and $\simplex_V$, $-\simplex_V$ which are (geometric) triangles.
The cone $\DCZG[K_3]$ itself is (a cone over) a bi-pyramid over a triangle: $\simplex_a$, $\simplex_b$, $\simplex_c$ are the vertices of the triangle, and $\simplex_V$, $-\simplex_V$ the two apices of the bi-pyramid.
Consequently, any deformation $\polytopeP$ of $\ZG[K_3]$ can be written (in a unique way) as $\polytopeP = \lambda_a \simplex_a +\lambda_b\simplex_b + \lambda_c\simplex_c \pm \lambda_V \simplex_V$ with $\lambda_a,\lambda_b,\lambda_c,\lambda_V\geq 0$.
This description is a triangulation of $\DCZG[K_3]$, see \Cref{fig:DCZG_K3}.

Note that $\pm\simplex_V$ \textbf{are not} graphical zonotopes themselves.

\begin{figure}[h]
    \centering
    \includegraphics[scale=0.8]{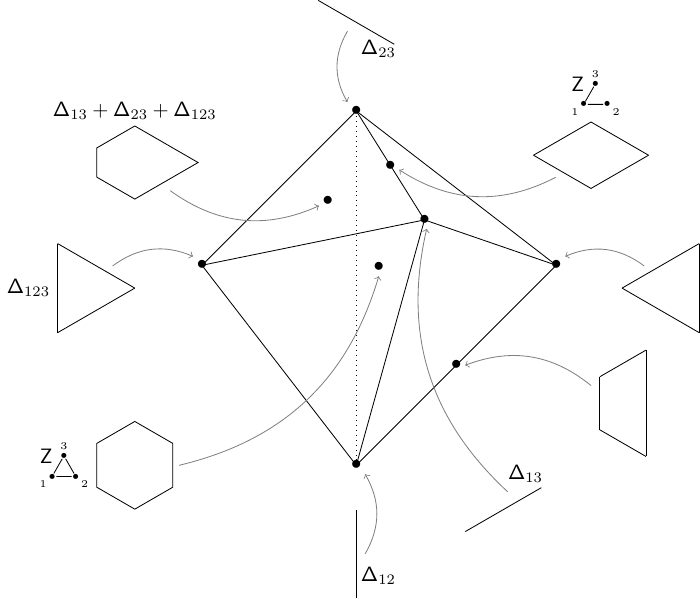}
    \caption{(\cite[Figure 2]{PPP2023gZono}). For $K_3$, the graphical zonotope $\ZG[K_3]$ is a regular hexagon, \ie the 2-dimensional permutahedron (bottom left).
    Its deformation cone $\DCZG[K_3]$ is 4-dimensional.
    We picture a 3-dimensional affine section of $\DCZG[K_3]$.
    The deformations of $\ZG[K_3]$ corresponding to some of the points of $\DCZG[K_3]$ are depicted.
    Especially, all points in the interior correspond to polytopes normally equivalent to $\ZG[K_3]$, while the above left polytope is Loday's associahedron.\vspace{-0.2cm}}
    \label{fig:DCZG_K3}
\end{figure}

\subsection{Triangle-free graphs}\label{ssec:Triangle_free}

Triangle-free graphs, \ie graphs with no induced $K_3$, has been addressed in \cite[Section 2.4]{PPP2023gZono}.

\begin{theorem}\label{thm:Triangle_free_DC_simplicial}\emph{(\cite[Corollary 2.9 \& 2.10]{PPP2023gZono}).}
The deformation cone $\DC[\ZG]$ is simplicial if and only if $G = (V, E)$ is triangle-free.
In this case, each face of this simplicial cone is (linearly equivalent to) $\DC[\polytope{Z}_H]$ for a sub-graph $H$ of $G$ (defined by a subset of edges).
\end{theorem}

This follows from \Cref{thm:DCZG_dim_and_nbr_facet}:
if $G$ is triangle-free, then $\Omega(G) = |E| = \sum_{e\in E} 2^{\left|\left\{t\in T ~;~ e\,\subseteq\, t\right\}\right|}$, so the dimension and number of facets of $\DCZG$ are equal.
Conversely, if $G$ has a triangle, then the number of facets of $\DCZG$ grows faster than its dimension (see the proof in \cite[Corollary 2.10]{PPP2023gZono}).
For the second part of the theorem, note that $\ZG[H]$ is a deformation of $\ZG$ when $H$ is a sub-graph of $G$, hence $\DCZG[H]$ is a face of $\DCZG$: this gives $2^{|E|}$ faces.


Motivated by \Cref{thm:Triangle_free_DC_simplicial}, it is natural to ask what happens when $G$ has triangles.
This was starting point of the present paper.
This last sub-section of these preliminaries is devoted to a second proof of \Cref{thm:Triangle_free_DC_simplicial}.
The one of \cite[Section 2.4]{PPP2023gZono} relies on the description of the deformation cone as the \emph{height} deformation cone, whereas the following one will make use of the \emph{edge-length} deformation cone.
This demonstrates how powerful this description can be.

\begin{proof}[Alternative proof of \Cref{thm:Triangle_free_DC_simplicial}]
Fix $G = (V, E)$ a triangle-free graph, and $\b \ell\in \DCe[\polytope{Z}_G]$.
We are going to prove that $\DCZG$ is simplicial.
As the edges of $\ZG$ are in bijection with the ordered partitions $(e, \rho)$ of $G$ with $e\in E$, $\rho\in \c O(\contrG[e])$, we label the coordinates of $\b\ell$ by $\ell_{e, \rho}$.

As $G$ is triangle-free, faces of dimension $2$ of $\ZG$ are parallelograms, according to \Cref{fig:table_small_dim_ZG}.
Consider two edges opposite in a 2-face $\polytopeF$: they correspond to $(e, \rho)$ and $(e, \rho')$ for $\rho, \rho'\in \c O(\contrG[e])$ linked by a flip in $\contrG[e]$.
The polygonal face equation of $\polytopeF$ ensures $\ell_{e,\rho} = \ell_{e, \rho'}$, as seen in \Cref{ssec:Parallelograms_and_hexagons}.

Fix an arc $e\in E$, consider all $\rho\in \c O(\contrG[e])$.
We know that $\ell_{e,\rho} = \ell_{e, \rho'}$ if $\rho, \rho'\in \c O(\contrG[e])$ are linked by a flip in $\contrG[e]$.
As the flip graph of $\contrG[e]$ is connected, $\ell_{e,\rho} = \ell_{e, \rho'}$ for all $\rho, \rho'\in \c O(\contrG[e])$.
We denote by $\ell_e$ this quantity.

It remains to prove that $\polytopeQ_{\b\ell}$ is (the translation of) $\polytope{R} := \sum_{e\in E}\ell_e \simplex_e$.
The polytope $\polytope{R}$ is normally equivalent to $\ZG[H]$ where $H$ is the sub-graph of $G$ whose arcs are all $e\in E$ with $\ell_e \ne 0$.
First note that $\Gamma_{\polytopeQ_{\b\ell}}$ and $\Gamma_{\polytope{R}}$ are isomorphic because they are obtained from $\Gamma_{\ZG}$ by contracting the edges of $\ZG$ of length $0$.
Pick a vertex $\b v$ of $\polytopeQ_{\b\ell}$, and its associated vertex $\b v'$ of $\polytope{R}$ (\eg obtained by choosing a generic direction $\b c\in \R^V$, then $\b v = \polytopeQ_{\b\ell}^{\b c}$ and $\b v' = \polytope{R}^{\b c}$).
To conclude, for any edge-path $\c P$ in $\Gamma_{\polytopeQ_{\b\ell}}$ starting at $\b v$, we have $\sum_{(e,\rho)\in \c P} \varepsilon_{e,\rho}^{\c P} \ell_{e,\rho} \b u_{e, \rho} = \sum_{(e,\rho)\in \c P} \varepsilon_{e,\rho}^{\c P} \ell_e \b u_{e, \rho}$ which is precisely the edge-path from $\b v'$ to the vertex of $\polytope{R}$ that correspond to the ordered partition $(e, \rho)$.
Hence, $\polytopeQ_{\b\ell} = \polytope{R} + (\b v - \b v')$.

We have proven that any deformation of $\ZG$ can be written as a (positive) Minkowski sum of $\simplex_e$ for $e\in E$.
Thus the rays of $\DC[\polytope{Z}_G]$ are exactly the simplices $\simplex_e$ for $e\in E$, and $\DC[\polytope{Z}_G]$ is simplicial (either one can use \Cref{thm:DCZG_dim_and_nbr_facet} to obtain $\dim\DC[\polytope{Z}_G] = |E|$, or one directly sees that it is impossible to have $\sum_{e\in A}\lambda_e\simplex_e = \sum_{e\in B}\lambda_e\simplex_e$ for $E=A\sqcup B$ and $\lambda_e \geq 0$ not all zeros).


Conversely, if $G$ has a triangle $t$, then $H = \restrG[t]$ is a sub-graph of $G$, and $\DCZG[H]$ is a face of $\DCZG$ by \Cref{thm:faces_of_DC_are_DC}.
As $\DCZG[H] \simeq \DCZG[K_3]$ is not simplicial, neither is $\DCZG$.
\end{proof}

%% file: Figures/Table_small_dimensional_ZG.tex
\begin{figure}[h]
\begin{center}
\begin{tabular}{ >{\centering\arraybackslash} m{1cm} >{\centering\arraybackslash} m{5cm} >{\centering\arraybackslash} m{3cm} >{\centering\arraybackslash} m{3cm} }
     $\dim \ZG$ & $G$ & $\ZG$ & Figure \\ \hline
     2 & \includegraphics[scale=0.4]{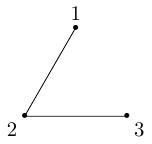} & Parallelogram & \includegraphics[scale=0.5]{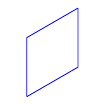}\\
     2 & \includegraphics[scale=0.4]{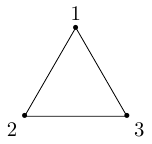} & Hexagon & \includegraphics[scale=0.4]{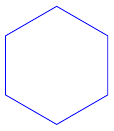} \\ \hline
     3 & \includegraphics[scale=0.4]{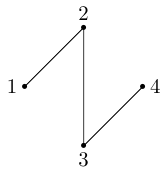} ~~ \includegraphics[scale=0.4]{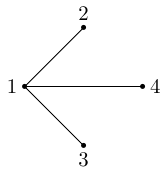} & Cube & \includegraphics[scale=0.5]{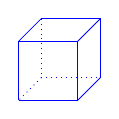} \\
     3 & \includegraphics[scale=0.4]{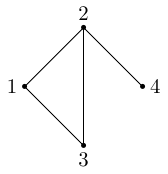} & Hexagonal prism & \includegraphics[scale = 0.375]{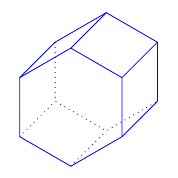} \\
     3 & \includegraphics[scale=0.4]{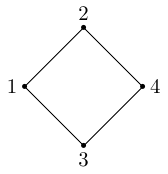} & Rhombic dodecahedron & \includegraphics[scale=0.375]{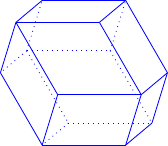} \\
     3 & \includegraphics[scale=0.4]{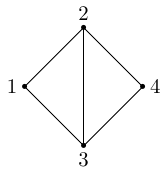} & Hexa-rhombic dodecahedron & \includegraphics[scale=0.3]{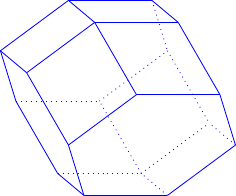} \\
     3 & \includegraphics[scale=0.4]{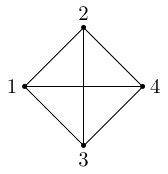} & Permutahedron & \includegraphics[scale=0.25]{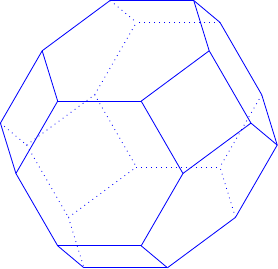}
\end{tabular}
\end{center}
\caption{Graphical zonotopes of dimension 2 and 3.
We only need to consider connected graphs on 4 nodes or less, as, for any non-connected graph $G$, the graphical zonotope $\ZG$ is normally equivalent to the one of any 1-sum of the connected components of $G$.}
\label{fig:table_small_dim_ZG}
\end{figure}

%% file: 2_K4free.tex
\section{$K_4$-free graphs}\label{sec:K4_free_graphs}

\subsection{Triangulation, rays and 2-faces of $\DCZG$}

We have seen that the deformation cone $\DCZG[K_3]$ for the graphical zonotope of a triangle has a particularly nice triangulation into two simplicial cones of dimension 4.
Indeed, recall that $\DCZG[K_3]$ is (a cone over) a bi-pyramid over a triangle, the vertices of the triangle being associated to segments $\simplex_e$ for $e\in E(K_3)$, and the apices of the bi-pyramid to triangles $\pm\simplex_V$.

We would like to extend this triangulation to the deformation cone $\DCZG$ for any graph $G$.
Especially, such a description would solve the question of the rays of $\DCZG$.
Unfortunately, it is clear that this is no more than a vain hope:
if we were able to do so, we would solve the case of $G = K_n$, and hence enumerate the rays of the submodular cone.
The later is an open problem since Edmonds \cite{Edmonds}, and it is considered ``impossible'' by Csirmaz and Csirmaz \cite{CsirmazCsirmaz-AttemptingTheImpossible} already for $n = 6$.
To be precise, already for $G = K_4$, there are rays of $\DCZG[K_4]$ that correspond to 3-dimensional polytopes (namely to a tetrahedron, to a regular octahedron, to a pyramid over a square, and to two other polytopes): describing $\DCZG[K_n]$ for $n \geq 4$ as (a cone over) ``\emph{something} over the simplex generated by $\simplex_e$ for $e\in \binom{[n]}{2}$" is for now out of reach.

Nevertheless, when $G$ is $K_4$-free, \ie no 4 nodes are pairwise linked by an arc, then the geometry of $\ZG$ is not intricate enough to allow all the wildness of the general case.
Indeed, we will prove that we can triangulate $\DCZG$ in a similar fashion as for $\DCZG[K_3]$.
Especially, the rays of $\DCZG$ are associated to segments $\simplex_e$ for $e\in E(G)$, and triangles $\pm\simplex_t$ for $t$ a triangle of $G$.

We denote $E$ the set of arcs of $G$, and $T$ its set of triangles.

\begin{definition}\label{def:SimplicialConesForTriangulation}
For a $K_4$-free graph $G$ and $(\varepsilon_t)_t\in \{-1, 1\}^T$, we denote by \defn{$\polytope{S}_{\varepsilon}$} the simplicial cone generated by the rays $\bigl(\simplex_e ~;~ e\in E\bigr)\,\cup\,\bigl(\varepsilon_t\simplex_t ~;~ t\in T\bigr)$.
\end{definition}

\begin{lemma}\label{lem:int_are_disjoint}
For any graph $G$, and $\varepsilon, \varepsilon'\in \{-1, +1\}^T$, if $\varepsilon\ne \varepsilon'$, then: $\text{int}\bigl(\polytope{S}_\varepsilon\bigr) \cap \text{int}\bigl(\polytope{S}_{\varepsilon'}\bigr) = \emptyset$.
\end{lemma}

\begin{proof}
Fix $\varepsilon\in \{-1, +1\}^T$ and $\polytopeP \in \text{int}\bigl(\polytope{S}_\varepsilon\bigr)$.
There exists $\lambda_e, \lambda_t > 0$ such that $\polytopeP = \sum_{e\in E} \lambda_e \simplex_e + \sum_{t\in T} \varepsilon_t\lambda_t\simplex_t$.
Fix $e\in E$, $t\in T$ with $e\subset t$, and let $f$ be the arc of $\contrG[e]$ that correspond to the contraction of the two other arcs of $t$.
Pick any orientation $\rho\in \c O(\contrG[e])$ such that $f$ is oriented towards $e$ (which is a node of $\contrG[e]$).
For $\b n \in \polytopeC_{e, \rho}$, the face optimizing $\b n$ is $\polytopeP^{\b n} = \lambda_e\simplex_e + \sum_{s\in T, ~e\subset s}\lambda_s(\varepsilon_s\simplex_s)^{\b n}$.
We have $(\varepsilon_t\simplex_t)^{\b n} = \simplex_e$ if $\varepsilon_t > 0$, and $(\varepsilon_t\simplex_t)^{\b n}$ is a point if $\varepsilon_t < 0$.
Hence, we can read the value of $\varepsilon_t$ from the edge lengths of $\polytopeP$:
let $\rho'\in\c O(\contrG[e])$ be the orientation obtained from $\rho$ by changing the orientation of $f$, and $\b n' \in \polytopeC_{e, \rho'}$.
If the length of the edge $\polytopeP^{\b n}$ is greater than the length of the edge $\polytopeP^{\b n'}$, then $\polytopeP\notin \text{int}\bigl(\polytope{S}_\delta\bigr)$ for all $\delta\in\{-1, +1\}^T$ with $\delta_t < 0$, and conversely.
Consequently, there is a unique $\delta\in\{-1, +1\}^T$ such that $\polytopeP\in \text{int}\bigl(\polytope{S}_\delta\bigr)$.
\end{proof}

\begin{theorem}\label{thm:DCZG_K4-free_graphs}
Let $G$ be a $K_4$-free graph, $E$ its set of arcs and $T$ its set of triangles.
The cone $\DCZG$ has dimension $|E| + |T|$ and is triangulated by the $2^{|T|}$ simplicial cones $\polytope{S}_{\varepsilon}$ for $\varepsilon \in\{-1, +1\}^T$:
we have $\DCZG = \bigcup_{\varepsilon \in \{-1, 1\}^T} \polytope{S}_{\varepsilon}$ with $\text{int}\bigl(\polytope{S}_{\varepsilon}\bigr)\cap \text{int}\bigl(\polytope{S}_{\varepsilon'}\bigr) = \emptyset$ for $\varepsilon\ne \epsilon'$.

Especially, the $|E| + 2|T|$ rays of $\DCZG$ are associated to the segments $\simplex_e$ for $e\in E$, and the triangles $\simplex_t$ and $-\simplex_t$ for $t\in T$.
\end{theorem}

The proof of this theorem is the main goal of the present section and is provided in \Cref{ssec:proof_of_thm_DCZG_K4-free}.
It requires some preparatory lemmas, especially a detailed understanding of the geometry of the hexa-rhombic dodecahedron, see \Cref{ssec:Hexa-rhombic_Dodecahedron}.
Note that the dimension of $\DCZG$ comes directly from \Cref{thm:DCZG_dim_and_nbr_facet}, as $\Omega(G) = |E| + |T|$ because maximal cliques are triangles.
The same theorem gives that the number of facets of $\DCZG$ is $\sum_{e\in E} 2^{|\{t\in T ~;~ e\,\subseteq\, t\}|}$, which can be arbitrarily high.

\begin{remark}
For a connected $K_4$-free graph $G$ on $n$ nodes, the dimension of $\ZG$ is $\dim \ZG = n - 1$.
However, according to \cref{thm:DCZG_K4-free_graphs}, any deformation of $\ZG$ can be described as a sum of polytopes of dimension $1$ and $2$, whatever the value of $n$.
Although having a description as a Minkowski sum of low dimensional polytopes is not a complete surprise for a generalized permutahedron (for instance $\ZG = \sum_{e\in E} \simplex_e$), most generalized permutahedra do not enjoy such an easy description.
Accordingly, such phenomenon witnesses a very constrained geometry.

The aim of \Cref{sec:High_dimensional_summands} is to discuss the fact that there exists graphs $G = (V, E)$ that are not $K_4$-free but are $K_5$-free, and which admit a deformation that is Minkowski indecomposable of dimension $|V| - 1$.
\end{remark}

As the cone $\DCZG[K_3]$ is (a cone over) a bi-pyramid over a triangle, one could think that, when $G$ is a $K_4$-free graph, $\DCZG$ can be written as (a cone over) a sequence of bi-pyramids taken over the initial simplex of vertices $\simplex_e$ for $e\in E$, with the pairs of apices $(\simplex_t, -\simplex_t)$ for $t\in T$.
The following \Cref{cor:DCZG_graph_is_bi-pyramid_graph} shows that the ``graph of $\DCZG$"\footnote{The graph of the polytope obtained by intersecting $\DCZG$ with an hyperplane, \ie the $2$-skeleton of $\DCZG$.} is indeed the graph of such a sequence of bi-pyramids.
However, $\DCZG$ can not be described as such a sequence of bi-pyramids in general:
take $G$ to be the \defn{bi-triangle graph}, that is the graph obtained by gluing two triangles $K_3$ along a common arc, then $|E| = 5$, $|T| = 2$, thus, by \Cref{thm:DCZG_dim_and_nbr_facet}, $\dim \DCZG = |E| + |T| = 7$ and $\DCZG$ has $4\times 2^1 + 1 \times 2^2 = 12$ facets, but taking $|T| = 2$ consecutive bi-pyramids over a simplex of dimension $|E| = 5$ would lead to $2\times 2\times 5 = 20$ facets.

\begin{corollary}\label{cor:DCZG_graph_is_bi-pyramid_graph}
Let $G$ be a $K_4$-free graph, $E$ its set of arcs, $T$ its set of triangles.
The number of 2-dimensional faces of $\DCZG$ is $\binom{|E|+2|T|}{2} - |T|$, and they are associated to (see \Cref{fig:table_2-dim_faces_DCZG}):
\begin{compactenum}[(i)]
\item parallelograms $\simplex_e + \simplex_f$ for $e, f\in E$, $e\ne f$,
\item trapezoids $\simplex_e + \simplex_t$ and $\simplex_e - \simplex_t$ for $e\in E$, $t\in T$ with $e\subseteq t$,
\item prisms over a triangle $\simplex_e + \simplex_t$ and $\simplex_e - \simplex_t$ for $e\in E$, $t\in T$ with $e\not\subset t$,
\item prismatoids\footnote{Note that there are two kinds of prismatoids, depending where $t$ and $t'$ share a common edge or not. These prismatoids are combinatorially equivalent but not normally equivalent, see \Cref{exmpl:DCZG_two_triangles}.} $\simplex_t + \simplex_{t'}$ and $\simplex_t + \simplex_{t'}$ for $t, t'\in T$ with $t\ne t'$.  
\end{compactenum}

Said differently, the only pairs of rays of $\DCZG$ that do not lie in a common 2-face are the pairs associated to $(\simplex_t, -\simplex_t)$ for $t\in T$.
\end{corollary}

\input{Figures/Table_2-dim_faces_DCZG}

\begin{proof}
By \Cref{thm:DCZG_K4-free_graphs}, we know that the rays of $\DCZG$ are $\simplex_e$ for $e\in E$ and $\simplex_t$, $-\simplex_t$ for $t\in T$.
Firstly, for $t\in T$, $\simplex_t + (-\simplex_t) = \ZG[K_3]$ is a regular hexagon whose deformation cone has dimension $4$ (see \Cref{ssec:Parallelograms_and_hexagons}), so the pair $(\simplex_t, -\simplex_t)$ does not lie in a common $2$-face of $\DCZG$.
It remains to prove that any other pair is associated to a 2-face of $\DCZG$.
This is easily done with a computer: for each polytope $\polytopeP$ at stake, we write the polygonal face equations and compute $\dim\DCe$.
There are only $4$ cases to check, as all parallelograms (respectively trapezoids, prisms over triangle, prismatoids) are equivalent up to change of coordinates.
Note that all these $2$-dimensional faces are well-known, as they appear as 2-dimensional faces of the submodular cone $\DCZG[K_4]$ (they also appear as $2$-dimensional faces of $\DCZG[G]$ for $G$ the bi-triangle graph).
\end{proof}

\subsection{Hexa-rhombic dodecahedra and hexagonal prisms}\label{ssec:Hexa-rhombic_Dodecahedron}

When $G$ is the \defn{bi-triangle graph}, that is the graph on $4$ nodes obtained by gluing two triangles $K_3$ on a common arc, then the graphical zonotope $\ZG$ is a \defn{hexa-rhombic dodecahedron} (also called elongated dodecahedron).

\Cref{fig:ZG_two_triangles} represent the hexa-rhombic dodecahedron (left), and its graph (right).
It has 12 facets: 8 parallelograms and 4 hexagons.
The hexagons come in 2 pairs of parallel hexagons.

Fix a deformation of $\ZG$ and consider two parallel hexagons $\polytopeH$ and $\polytopeH'$ with their step $\delta_\polytopeH$, $\delta_{\polytopeH'}$ (see \Cref{def:step_hexagon}).
Recall that the polygonal face equation ensures that edges that are opposite in a parallelogram have the same length.
Consequently, $\delta_\polytopeH = \delta_{\polytopeH'}$ because the corresponding edges of $\polytopeH$ and $\polytopeH'$ share the same lengths, see \Cref{fig:ZG_two_triangles} (right: edges with the same colour have the same length in any deformation of $\ZG$).

Now, we can extend this result to any $K_4$-free graph.
When $G$ is $K_4$-free, hexagonal faces of $\ZG$ are associated to ordered partitions of the form $(t, \rho)$ where $t$ is a triangle of $G$, and $\rho\in\c O(\contrG[t])$.
We denote by $\delta_{t, \rho}$ the step of the hexagon associated to $(t, \rho)$.

\begin{lemma}\label{lem:delta_t}
Let $G$ be a $K_4$-free graph, $t$ a triangle of $G$.
For $\rho, \rho'\in \c O(\contrG[t])$ one has $\delta_{t, \rho} = \delta_{t, \rho'}$ for any deformation of $\ZG$.
Consequently, we name $\delta_t$ this quantity.
\end{lemma}

\begin{proof}
Fix $G$ a $K_4$-free graph, and $t$ a triangle of $G$.

Pick $\rho\in\c O(\contrG[t])$, and an arc $\overline{e}$ of $\contrG[t]$ such that the orientation $\rho'$, obtained from $\rho$ by changing the orientation of $\overline{e}$, is acyclic.
As $G$ is $K_4$-free, there are two possibilities: either $\overline{e}$ comes from the contraction of two arcs $e$ and $f$ of $G$ that form a triangle with one of the arcs of $t$; or $\overline{e}$ correspond to an original arc $g$ of $G$.

In the first case, $\delta_{t, \rho} = \delta_{t, \rho'}$ for any deformation of $\ZG$ because the hexagonal faces associated to $(t, \rho)$ and $(t, \rho')$ are opposite facets of the hexa-rhombic dodecahedron $\contrG[t\,\cup\, e\,\cup\, f]$, see \Cref{fig:table_small_dim_ZG}.

In the second case, $\delta_{t, \rho} = \delta_{t, \rho'}$ for any deformation of $\ZG$ because the hexagonal faces associated to $(t, \rho)$ and $(t, \rho')$ are opposite facets of the hexagonal prism $\contrG[t\,\cup\, g]$, see \Cref{fig:table_small_dim_ZG}.

Now consider all $\rho\in \c O(\contrG[t])$.
We know that $\delta_{t, \rho} = \delta_{t, \rho'}$ if $\rho, \rho'\in \c O(\contrG[t])$ are linked by a flip in $\contrG[t]$.
As the flip graph of $\contrG[t]$ is connected, $\delta_{t, \rho} = \delta_{t, \rho'}$ for all $\rho, \rho'\in \c O(\contrG[t])$.
\end{proof}

\begin{figure}
\begin{subfigure}[b]{0.49\textwidth}
\centering
\includegraphics[scale=1.25]{Figures/ZG/Two_triangles.pdf}
\end{subfigure}
\begin{subfigure}[b]{0.49\textwidth}
\centering
\includegraphics[scale=1.7]{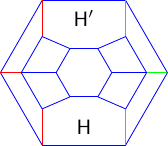}
\end{subfigure}
\caption{(Left) Hexa-rhombic dodecahedron.
(Right) The graph of the hexa-rhombic dodecahedron, obtained as a Schlegel projection on one of its hexagonal facets.
Edges with the same color have the same length, and the two labeled hexagons are parallel in the hexa-rhombic dodecahedron.
Note that there are other edges with the same length, but we do not use them.}
\label{fig:ZG_two_triangles}
\end{figure}

\begin{example}\label{exmpl:DCZG_two_triangles}
To conclude this sub-section, we can explicitly compute the deformation cone of the hexa-rhombic dodecahedron: $\DCZG$ for $G$ the bi-triangle graph.
We have seen that $|E| = 5$ and $|T| = 2$, $\dim\DCZG = \Omega(G) = |E| + |T| = 7$, and $\DCZG$ has $\sum_{e \in E} 2^{|\{t\in T~;~ e\,\subseteq\, t\}|}~=~12$ facets.
Thanks to \Cref{thm:DCZG_K4-free_graphs}, $\DCZG$ has $|E|+2|T| = 9$ rays, associated to segments $\simplex_e$ for $e\in E$, and triangles $\pm\simplex_t$ for $t\in T$.
With \Cref{cor:DCZG_graph_is_bi-pyramid_graph}, we obtain that $\DCZG$ has $\binom{9}{2} - 2 = 34$ faces of dimension 2.
The whole $f$-vector can be obtained via a computer experiment:
$$f_{\DCZG} = (9, 34, 68, 75, 44, 12, 1)$$

It would be impossible to picture here the polytopes associated to the $243$ different faces of $\DCZG$.
However, a lot of these polytopes turn out to be equivalent up to symmetries of the braid fan (\ie one is obtained from the other by permuting the coordinates and applying central symmetry).
For instance, all $5$ segments $\simplex_e$ for $e\in E$ are the same up to symmetries, and all $4$ triangles $\pm\simplex_t$ for $t\in T$ are the same.
After quotienting by these equivalence relations, we obtain the number of classes of faces per dimension:
$\overline{f}_{\DCZG} = (2, 6, 10, 13, 9, 3, 1)$.

We can even quotient a bit more.
Among the $6$ faces of dimension $2$, two of them are associated to parallelograms: $\simplex_e + \simplex_f$ for $e$, $f$ arcs of the same triangle, and $\simplex_e + \simplex_f$ (which is a square) for $e$, $f$ not belonging to the same triangle.
Hence, there are faces of $\DCZG$ associated to combinatorially isomorphic polytopes that are not normally equivalent.
For a matter of space, we draw the $24$ non-combinatorially-isomorphic polytopes associated to faces of $\DCZG$, see \Cref{fig:All_face_DCZG_two_triangles}.
\end{example}

\input{Figures/Table_all_faces_DCZG_two_triangles}

\subsection{Proof of the triangulation of $\DCZG$ for $K_4$-free graph}\label{ssec:proof_of_thm_DCZG_K4-free}


\begin{proof}[Proof of \Cref{thm:DCZG_K4-free_graphs}]
Let $G$ be a $K_4$-free graph, $E$ its set of arcs and $T$ its set of triangles.
Fix $\b \ell\in \DCe[\polytope{Z}_G]$, and label its coordinates by $\ell_{e, \rho}$ with $e\in E$ and $\rho\in \c O(\contrG[e])$.

We are going to prove that $\polytopeQ_{\b\ell}$ is a translation of $\sum_{e\in E} \omega(e) \simplex_e + \sum_{t\in T}\omega(t) \simplex_t$, for some $\omega(e)\geq0$ and $\omega(t)\in\R$.
This will show that $\polytopeQ_{\b\ell}$ is in the cone generated by $\simplex_e$ for $e\in E$ and $\varepsilon_t\simplex_t$ for $t\in T$ with $\varepsilon_t = 1$ if $\omega(t) > 0$, or $\varepsilon_t = -1$ if $\omega(t) < 0$.
As $\dim\DCZG = |E| + |T|$, the latter is a simplicial cone (it has as many rays as dimensions), and we triangulate $\DCZG$ by the collection of $2^{|T|}$ cones: $\cone\bigl(\{\simplex_e~;~e\in E\}\cup \{\varepsilon_t \simplex_t~;~t\in T\}\bigr)$ for $(\varepsilon_t)_{t\in T} \in \{-1,+1\}^T$.
\Cref{lem:int_are_disjoint} guaranties that the interior of theses cones are indeed disjoint.


Fix $e\in E$, and choose $\rho\in \c O(\contrG[e])$ that minimizes $\ell_{e, \rho}$ (among all $\ell_{e, \rho'}$ for $\rho'\in \c O(\contrG[e])$).

We set $\omega(e) := \ell_{e, \rho}$.
If $e$ is an arc of a triangle $t$, then let $\overline{f}$ be the arc of $\contrG[e]$ that is the contraction of the two other arcs of $t$.
If $\overline{f}$ is oriented by $\rho$ towards $e$ (which is a vertex of $\contrG[e]$), then we set $\omega(t) := -\delta_{t}$~; if $\overline{f}$ is oriented by $\rho$ away from $e$, then we set $\omega(t) := \delta_{t}$.

It remains to prove two things: first $\omega(t)$ is well-defined (\ie does not depend on $e$), and second that $\polytopeQ_{\b\ell}$ is a translation of $\sum_{e\in E} \omega(e) \simplex_e + \sum_{t\in T}\omega(t) \simplex_t$ for these chosen values of $\omega(e)$ and $\omega(t)$.


Fix $t\in T$ and consider its arcs $a, b, c\in E$, together with $\rho\in\c O(\contrG[a])$ that minimizes $\ell_{a, \rho'}$.
Let $\overline{b}$ be the arc in $\contrG[a]$ obtained by contracting $b$ (it is also $\overline{c}$).
To ease notation, suppose that $\overline{b}$ is oriented by $\rho$ towards $a$ (which is a node of $\contrG[a]$), the other case is symmetric.
Let $\overline{\rho}\in\c O(\contrG[t])$ obtained by contracting $\overline{b}$ in $\rho$.
In the hexagonal face associated to $(t, \overline{\rho})$, the edge opposite to $(a, \rho)$ is longer than the edge associated to $(a, \rho)$, by construction (the difference of length is $\delta_{t, \overline{\rho}}$), this edge is associated to $(a, \rho')$ where $\rho'$ is obtained from $\rho$ by changing the orientation of $\overline{b}$.
Besides, the two edges associated to $c$ in the hexagonal face $(t, \overline{\rho})$ are $(c, \sigma)$ and $(c, \sigma')$ where $\sigma, \sigma'\in\c O(\contrG[c])$ agree with $\overline{\rho}$ on $\contrG[t]$, and $\sigma$ orients the remaining arc of $t$ towards $c$, whereas $\sigma'$ orients it away from $c$.
From the polygonal face equation of the hexagonal face $(t,\overline{\rho})$, we obtain that $\ell_{a,\rho} - \ell_{a,\rho'} = \ell_{c,\sigma} - \ell_{c,\sigma'}$, see \Cref{ssec:Parallelograms_and_hexagons}.
Hence $\ell_{c,\sigma} - \ell_{c,\sigma'} = \delta_{t,\overline{\rho}} \geq 0$ and finally $\ell_{c,\sigma} > \ell_{c,\sigma'}$.
This proves that $\omega(t)$ is well-defined: if $\min_\rho \ell_{a, \rho}$ is attained when $\rho$ orients the remaining arc of $t$ towards $a$, then $\min_\sigma \ell_{c, \sigma}$ is attained when $\sigma$ orients the remaining arc of $t$ towards $c$.


It remains to prove that $\polytopeQ_{\b\ell}$ is (the translation of) $\polytope{R} = \sum_{e\in E} \omega(e) \simplex_e + \sum_{t\in T}\omega(t) \simplex_t$.
Fix $e\in E$ and $\rho\in\c O(\contrG[e])$, and let $\b n\in \polytopeC_{e, \rho}$ be an outer normal vector of the edge of $\ZG$ associated to the ordered partition $(e, \rho)$, see \Cref{prop:Normal_fan_ZG}.
We have that $\polytopeQ_{\b\ell}^{\b n}$ (recall that $\polytopeQ_{\b\ell}^{\b n}$ is the face of $\polytopeQ_{\b\ell}$ that maximizes the scalar product against $\b n$) is a segment of direction $\simplex_e$ and length $\ell_{e, \rho}$ (if $\ell_{e,\rho} = 0$, then it is a vertex), and
$\polytope{R}^{\b n} = \sum_{f\in E} \omega(f) \simplex_f^{\b n} + \sum_{t\in T}\omega(t) \simplex_t^{\b n}$.
On the one hand, $\simplex_f^{\b n}$ equals $\simplex_e$ if $f = e$, and is a point else way.
On the other hand, $\omega(t)\simplex_t^{\b n}$ equals $\omega(t)\simplex_e$ if $e$ is an arc of $t$ and either $\omega(t) > 0$ and $\rho$ orients the remaining arc of $t$ towards $e$, or $\omega(t) < 0$ and $\rho$ orients the remaining arc of $t$ away from $e$.
Else way $\omega(t)\simplex_t^{\b n}$ is a point.
We say that $\omega(t)$ ``agrees" with $\rho$ when $e$ is an arc of $t$ and either $\omega(t) > 0$ and $\rho$ orients the remaining arc of $t$ towards $e$, or $\omega(t) < 0$ and $\rho$ orients the remaining arc of $t$ away from $e$.
Hence, $\polytope{R}^{\b n}$ is an edge in direction $\simplex_e$ with length:\vspace{-0.24cm}
$$\omega(e, \rho) := \omega(e) + \sum_{\substack{t\in T,~~ e\in t \\ \omega(t) \text{ agrees with } \rho}} |\omega(t)|$$

\vspace{-0.2cm}
To prove that $\ell_{e, \rho} = \omega(e,\rho)$ for all $\rho\in \c O(\contrG[e])$, we proceed by induction on the number of triangles $t\in T$ (containing $e$) such that $\omega(t)$ agrees with $\rho$.

If $\ell_{e, \rho}$ is minimal (among $\ell_{e, \rho'}$), then no $\omega(t)$ agrees with $\rho$ by construction of $\omega(t)$ and thus $\omega(e, \rho) = \omega(e) = \ell_{e, \rho}$ by construction of $\omega(e)$.

Furthermore, suppose $\omega(e, \rho) = \ell_{e, \rho}$, and let $\rho'\in\c O(\contrG[e])$ be obtained from $\rho$ by changing the orientation of one arc $f$.
If $f$ and $e$ do not belong to a triangle, then $\ell_{e, \rho'} = \ell_{e, \rho}$ (because they are opposite edges in a parallelogram), and $\omega(e, \rho') = \omega(e, \rho)$ (because changing the orientation of $f$ does not change the agreement of any $\omega(t)$ with $\rho$).
If $f$ and $e$ belong to a triangle $t$ and $\omega(t)$ agrees with $\rho'$ but disagree with $\rho$, then $|\ell_{e, \rho'} - \ell_{e, \rho}| = \delta_t$ by \Cref{lem:delta_t}, and $\omega(e, \rho') - \omega(e, \rho) = |\omega(t)| = \delta_t$.

We get by induction that $\ell_{e, \rho} = \omega(e, \rho)$ for all $e\in E$ and $\rho \in\c O(\contrG[e])$ (note that, to conduct this induction, we need that the graph of flips of $\c O(\contrG[e])$ is connected, which is the case).
As the edges of $\polytopeQ_{\b\ell}$ and $\polytope{R}$ have same directions and same lengths, one is the translate of the other.
\end{proof}











%% file: Figures/Table_2-dim_faces_DCZG.tex
\begin{figure}
\begin{center}
\begin{tabular}{ >{\centering\arraybackslash} m{3cm} >{\centering\arraybackslash} m{3cm} >{\centering\arraybackslash} m{3cm} >{\centering\arraybackslash} m{3cm} }
     Parallelogram & Trapezoid & Prism over triangle & Prismatoid \\ 
     \includegraphics[scale=1]{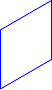} & \includegraphics[scale=0.8]{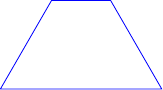} & \includegraphics[scale=0.8]{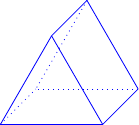} & \includegraphics[scale=0.8]{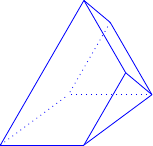} \\
     $\simplex_e + \simplex_f$ & $\simplex_e \pm \simplex_t$ & $\simplex_e \pm \simplex_t$ & $\simplex_t \pm \simplex_{t'}$ \\
     with $e\ne f$ & with $e\subseteq t$ & with $e\not\subset t$ & with $t \ne t'$
\end{tabular}
\end{center}
\caption{The polytopes associated to the 2-dimensional faces of $\DCZG$ for $G$ a $K_4$-free graph.}
\label{fig:table_2-dim_faces_DCZG}
\end{figure}

%% file: Figures/Table_all_faces_DCZG_two_triangles.tex
\begin{figure}[hb]
\begin{center}
\begin{tabular}{ >{\centering\arraybackslash} m{2.5cm} >{\centering\arraybackslash} m{2.5cm} >{\centering\arraybackslash} m{2.5cm} >{\centering\arraybackslash} m{2.5cm} >{\centering\arraybackslash} m{2.5cm} }
    \includegraphics[scale=1]{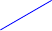} & \includegraphics[scale=1]{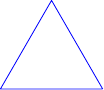} & \includegraphics[scale=1]{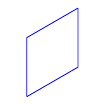} & \includegraphics[scale=0.9]{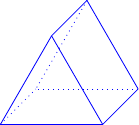} & \includegraphics[scale=0.8]{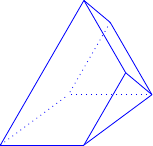} \\
    1, 5 & 1, 4 & 2, 22 & 2, 8 & 2, 4 \\ \\
    \includegraphics[scale=0.8]{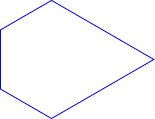} & \includegraphics[scale=1.1]{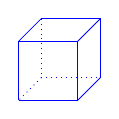} & \includegraphics[scale=1.1]{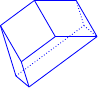} & \includegraphics[scale=1.25]{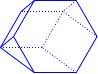} & \includegraphics[scale=1]{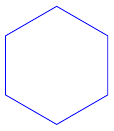} \\
    3, 12 & 3, 36 & 3, 16 & 3, 4 & 4, 2 \\ \\
    \includegraphics[scale=0.9]{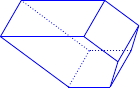} & \includegraphics[scale=1.5]{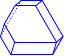} & \includegraphics[scale=1.65]{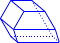} & \includegraphics[scale=1.6]{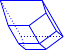} & \includegraphics[scale=1.6]{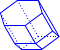} \\
    4, 40 & 4, 8 & 4, 8 & 4, 8 & 4, 8 \\ \\
    \includegraphics[scale=0.7]{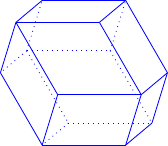} & \includegraphics[scale=0.7]{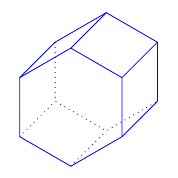} & \includegraphics[scale=1.65]{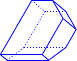} & \includegraphics[scale=1.6]{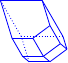} & \includegraphics[scale=1.6]{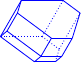} \\
    4, 1 & 5, 8 & 5, 8 & 5, 8 & 5, 16 \\ \\
    \includegraphics[scale=1.4]{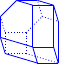} & \includegraphics[scale=1.3]{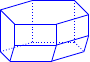} & \includegraphics[scale=1.3]{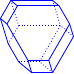} & \includegraphics[scale=0.55]{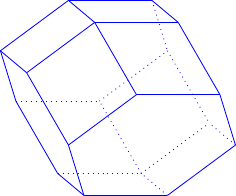} &  \\
    5, 4 & 6, 8 & 6, 4 & 7, 1 & 
\end{tabular}
\end{center}
\caption{All polytopes that are associated to faces of $\DCZG$ for $G$ the graph obtained by gluing two triangles $K_3$ along a common arc.
Under each polytope $\polytopeP$ are indicated $k, m$ where $k$ is the dimension of the face of $\DCZG$ that $\polytopeP$ is associated to, and $m$ is the number of faces of $\DCZG$ that are associated to (a polytope combinatorially isomorphic to) $\polytopeP$.}
\label{fig:All_face_DCZG_two_triangles}
\end{figure}

%% file: 3_WithK4.tex
\section{Graphs with high dimensional indecomposable summands}\label{sec:High_dimensional_summands}

If $G$ is a triangle-free graph, then the rays of $\DCZG$ are associated to segments (\Cref{thm:Triangle_free_DC_simplicial}); and if $G$ is a $K_4$-free graph, then the rays of $\DCZG$ are associated to segments and triangles (\Cref{thm:DCZG_K4-free_graphs}).
Morally, regarding Minkowski sums, the ``building blocks" of $\ZG$ are low-dimensional in these cases.
One would expect this to be a general phenomenon and may dare to state that: for any graph $G$, the dimension of any polytope associated to a ray of $\DCZG$ is at most the size of a maximal clique of $G$ minus 1.
This conjecture is false, see \Cref{exmp:counter-example1,exmp:counter-example2}.
Note that $\DCZG$ has rays associated to polytopes $\simplex_K$ for $K\subseteq V$ an induced clique of $G$, in particular the choice of $K$ among maximal cliques of $G$ shows that the above bound is a lower bound.

After numerous computational experiments, we propose the following conjecture:

\begin{conjecture}\label{conj:with_K4}
For all $n \geq 4$, there exists a graph $G$ on $n$ nodes that is $K_5$-free (but not $K_4$-free) such that there is a polytope $\polytopeP$ associated to a ray of $\DCZG$ with $\dim \polytopeP = n - 1$.
\end{conjecture}


\begin{example}\label{exmp:counter-example1}
\textit{Good candidates for proving the conjecture: graph of $\cyc(3, n)$}

Consider the graph $G_n$ on $n$ nodes obtained by taking a path $P_{n-2}$ on $n-2$ nodes $1,\dots,n-2$ and a path $P_2$ on $2$ nodes $0$ and $n-1$, and adding all the arcs $(0, i)$ and $(i, n-1)$ for $i\in [n-2]$.
This is the graph of the 3-dimensional cyclic polytope on $n$ vertices.
The graph $G_n$ is $K_5$-free, $\ZG[G_n]$ is an $(n-1)$-dimensional zonotope, and a computer check allows to prove that there exists a polytope $\polytopeP$ associated to a ray $\DCZG[G_n]$ with $\dim \polytopeP = n-1$, for all $n \leq 8$.

In particular, for $n = 5$, this gives a $K_5$-free graph with a $4$ dimensional Minkowski indecomposable deformation of $\ZG[G_n]$.
Note that there seem to be many such polytopes $\polytopeP$.
For instance, for $n = 8$, there are $6 450$ polytopes $\polytopeP$ associated to rays of $\DCZG[G_n]$ with $\dim\polytopeP = 7$.
Moreover, if we compute the number of polytopes $\polytopeP$ associated to a rays of $\DCZG[G_n]$ with $\dim\polytopeP = d$, we get:
\begin{center}
\begin{tabular}{c|ccccccc}
    $n\downarrow \,\backslash\, d\to$ & 1 & 2 & 3 & 4 & 5 & 6 & 7 \\ \hline
    4 & 6 & 8 & 23 & & & & \\
    5 & 9 & 14 & 46 & 96 & & & \\
    6 & 12 & 20 & 69 & 192 & 378 & & \\
    7 & 15 & 26 & 92 & 288 & 756 & 1542 & \\
    8 & 18 & 32 & 115 & 384 & 1134 & 3084 & 6450 \\
    \vdots &\vdots &\vdots &\vdots &\vdots &\vdots &\vdots &\vdots \\
    $n$ & $3(n-2)$ & $6n-16$ & $23(n-3)$ & $96(n-4)$ & $378(n-5)$ & $1542(n-6)$ & ? \\
\end{tabular}
\end{center}

The first lines of this table contain explicit values we computed, while the last line contains the formula we conjecture from these data.
These formulas seem particularly nice even though we do not have any clue on how to prove them.
What is immediate, is that $3(n-2)$ is the number of arcs of $G_n$, and $6n-16$ is twice the number of triangles of $G_n$.
It is easy to prove that $\simplex_e$ for $e$ an arc of $G_n$, and $\pm\simplex_t$ for $t$ a triangle of $G_n$ are associated to rays of $\DCZG[G_n]$ and are the only polytopes $\polytopeP$ associated to rays of $\DCZG[G_n]$ with $\dim\polytopeP \leq 2$.
The hard task starts in dimension 3.
\end{example}

\begin{example}\label{exmp:counter-example2}
\textit{Bad candidates: wedge sum of multiple $K_4$}

Consider the graph $G_n$ on $n$ nodes obtained by gluing $n-3$ copies of $K_4$ on a common triangle.
The maximal cliques of $G_n$ are $\bigl(1, 2, 3, k\bigr)$ for $k \in\{4,  n\}$.
Then, $G_n$ is $K_5$-free and $\ZG[G_n]$ is an $(n-1)$-dimensional zonotope.
A computer check allows to prove that there exists a polytope $\polytopeP$ associated to a ray $\DCZG[G_n]$ with $\dim \polytopeP = n-1$, for $n \in\{3, 4, 5\}$.
But for $n = 6$, there is none!
\end{example}

%% file: 4_Open_Questions.tex
\section{Open questions}


\paragraph{$\DCZG$ for not $K_4$-free graphs}
If $G$ is not $K_4$-free, it seems there are rays of $\DCZG$ associated to polytopes of dimension $\dim\ZG$.
Our experiments show \Cref{conj:with_K4} holds for $n \leq 8$.
To prove of this conjecture one should construct high dimensional indecomposable polytope whose normal fan is supported by the graphical arrangement $\c H_G$.
Besides, it is interesting to determine, the \defn{Minkowski dimension $\dim_\text{M} \ZG$} of the graphical zonotope of $G$, \ie the maximal dimension of a polytope associated to a ray of $\DCZG$.
We have $\Omega(G) \leq 1+\dim_\text{M} \ZG \leq |V|$.
If $G$ is $K_3$-free, then $\dim_\text{M}\ZG = 1$.
If $G$ is $K_4$-free, then $\dim_\text{M}\ZG = 2$.

\paragraph{$f$-vector of $\DCZG$ for $K_4$-free graph}
\Cref{thm:DCZG_K4-free_graphs,cor:DCZG_graph_is_bi-pyramid_graph} allows us to compute the values $f_1 = |E| + 2|T|$, and $f_2 = \binom{f_1}{2} - |T|$.
As we provide a triangulation of $\DCZG$, one may hope for working out the full $f$-vector of $\DCZG$ in this case, even though this is far from being immediate.
We did that for the bi-triangle graph in \Cref{exmpl:DCZG_two_triangles}, but the general case is open.

Besides, one may use tools from extremal graph theory such as the bound on the number of triangles of a $K_4$-free graph on $n$ vertices proven by Eckhoff \cite{Eckhoff-TriangleInK4freeGraph}, in order to produce faces of the submodular cone whose $f$-vector satisfy prescribed inequalities.

\paragraph{Generalized permutahedra of low Minkowski dimension}
The most surprising fact about $\DCZG$ for $K_4$-free graphs may not be that we are able to triangulate it efficiently, but that all the rays of $\DCZG$ are associated to 1- and 2-dimensional polytopes.
More than the $K_4$-freeness of $G$, the main reason is that the hexagonal faces in $\ZG$ are ``separated enough'' and can be dealt with independently.
We do not precise this here, but note that in the 3-dimensional permutahedron $\ZG[K_4]$ some hexagonal faces share an edge, and $\DCZG[K_4]$ has rays associated to 3-dimensional polytopes.

Hence, the author thinks it would be fruitful to properly define a notion of ``generalized permutahedra having \textit{well-separated} pentagons and hexagons".
\Cref{lem:delta_t} could extend to such generalized permutahedra $\polytopeP$, and thus one may prove that $\dim_\text{M} \polytopeP \leq 2$.
Said differently:
characterize the (polytopes associated to the) faces of the submodular cone with rays $\simplex_X$ for $|X| \in \{2, 3\}$.

Firstly, one can try to construct a class of nestohedra.
Proposition 3.29 of \cite{PPP2023Nesto} characterizes the nestohedra $\polytope{N}$ for which $\DC[\polytope{N}]$ is simplicial, and gives the rays of $\DC[\polytope{N}]$ in this case, using the same method as \cite[Corollary 2.9]{PPP2023gZono}: what about nestohedra $\polytope{N}$ with small $\dim_\text{M} \polytope{N}$?

%% file: main.bbl
\newcommand{\etalchar}[1]{$^{#1}$}
\begin{thebibliography}{BMCLD{\etalchar{+}}24}

\bibitem[AA17]{AguiarArdila}
Marcelo Aguiar and Federico Ardila.
\newblock Hopf monoids and generalized permutahedra.
\newblock {\em Memoirs of the American Mathematical Society}, 2017.

\bibitem[ABD10]{ArdilaBenedettiDoker}
Federico Ardila, Carolina Benedetti, and Jeffrey Doker.
\newblock Matroid polytopes and their volumes.
\newblock {\em Discrete Comput.~Geom.}, 43(4):841--854, 2010.

\bibitem[ACEP20]{ACEP-DeformationsCoxeterPermutahedra}
Federico Ardila, Federico Castillo, Christopher Eur, and Alexander Postnikov.
\newblock Coxeter submodular functions and deformations of {C}oxeter permutahedra.
\newblock {\em Advances in Mathematics}, 365:107039, 36, 2020.

\bibitem[AHBHY18]{ArkaniHamedBaiHeYan}
Nima Arkani-Hamed, Yuntao Bai, Song He, and Gongwang Yan.
\newblock Scattering forms and the positive geometry of kinematics, color and the worldsheet.
\newblock {\em J. High Energy Phys.}, (5):096, front matter+75, 2018.

\bibitem[AHBL17]{ArkaniHamedBaiLam-PositiveGeometries}
Nima Arkani-Hamed, Yuntao Bai, and Thomas Lam.
\newblock Positive geometries and canonical forms.
\newblock {\em J. High Energy Phys.}, (11):039, front matter+121, 2017.

\bibitem[AHT14]{ArkaniHamedTrnka-Amplituhedron}
N.~Arkani-Hamed and J.~Trnka.
\newblock The amplituhedron.
\newblock {\em J. High Energy Phys.}, 2014.

\bibitem[APR21]{AlbertinPilaudRitter}
Doriann Albertin, Vincent Pilaud, and Julian Ritter.
\newblock Removahedral congruences versus permutree congruences.
\newblock {\em Electron. J. Combin.}, 28(4):Paper No. 4.8, 38, 2021.

\bibitem[BLS{\etalchar{+}}99]{OrientedMatroids}
Anders Bj{\"o}rner, Michel {Las Vergnas}, Bernd Sturmfels, Neil White, and G{\"u}nter~M. Ziegler.
\newblock {\em Oriented matroids}, volume~46 of {\em Encyclopedia of Mathematics and its Applications}.
\newblock Cambridge University Press, Cambridge, second edition, 1999.

\bibitem[BMCLD{\etalchar{+}}24]{BazierMatteDouvilleMousavandThomasYildirim}
Véronique Bazier-Matte, Nathan Chapelier-Laget, Guillaume Douville, Kaveh Mousavand, Hugh Thomas, and Emine Yildirim.
\newblock {ABHY} {A}ssociahedra and {N}ewton polytopes of {F}-polynomials for cluster algebras of simply laced finite type.
\newblock {\em Journal of the London Mathematical Society}, 109(1):e12817, 2024.

\bibitem[CC25]{CsirmazCsirmaz-AttemptingTheImpossible}
Elod~P. Csirmaz and Laszlo Csirmaz.
\newblock Attempting the impossible: Enumerating extremal submodular functions for $n = 6$.
\newblock {\em Mathematics}, 13(1), 2025.

\bibitem[CDG{\etalchar{+}}22]{CastilloDoolittleGoecknerRossYing-MinkowskiSummandCube}
Federico Castillo, Joseph Doolittle, Bennet Goeckner, Michael~S. Ross, and Li~Ying.
\newblock Minkowski summands of cubes.
\newblock {\em Bull. Lond. Math. Soc.}, pages 996--1009, 2022.

\bibitem[CFZ02]{ChapotonFominZelevinsky}
Fr{\'e}d{\'e}ric Chapoton, Sergey Fomin, and Andrei Zelevinsky.
\newblock Polytopal realizations of generalized associahedra.
\newblock {\em Canad. Math. Bull.}, 45(4):537--566, 2002.

\bibitem[CL20]{CastilloLiu2020}
Federico Castillo and Fu~Liu.
\newblock Deformation cones of nested braid fans.
\newblock {\em Int. Math. Res. Not. IMRN}, 2020.

\bibitem[DK00]{DanilovKoshevoy2000}
Vladimir~I. Danilov and Gleb~A. Koshevoy.
\newblock Cores of cooperative games, superdifferentials of functions, and the {M}inkowski difference of sets.
\newblock {\em J. Math. Anal. Appl.}, 247(1):\mbox{pp. 1--14}, 2000.

\bibitem[Eck99]{Eckhoff-TriangleInK4freeGraph}
J\"urgen Eckhoff.
\newblock The maximum number of triangles in a {$K_4$}-free graph.
\newblock {\em Discrete Math.}, 194(1-3):95--106, 1999.

\bibitem[Edm70]{Edmonds}
Jack Edmonds.
\newblock Submodular functions, matroids, and certain polyhedra.
\newblock In {\em Combinatorial {S}tructures and their {A}pplications ({P}roc. {C}algary {I}nternat. {C}onf., {C}algary, {A}lta., 1969)}, pages 69--87. Gordon and Breach, New York, 1970.

\bibitem[Fuj05]{SubmodularFunctionsOptimization}
Satoru Fujishige.
\newblock {\em Submodular functions and optimization}, volume~58 of {\em Annals of Discrete Mathematics}.
\newblock Elsevier B. V., Amsterdam, second edition, 2005.

\bibitem[GKZ08]{GelfandKapranovZelevinsky}
Israel Gelfand, Mikhail Kapranov, and Andrei Zelevinsky.
\newblock {\em Discriminants, resultants and multidimensional determinants}.
\newblock Modern Birkh\"auser Classics. Birkh\"auser Boston Inc., Boston, MA, 2008.
\newblock Reprint of the 1994 edition.

\bibitem[JKS22]{JoswigKlimmSpitz2022}
Michael Joswig, Max Klimm, and Sylvain Spitz.
\newblock Generalized permutahedra and optimal auctions.
\newblock {\em SIAM Journal on Applied Algebra and Geometry}, 6(4), 2022.

\bibitem[McM73]{McMullen-typeCone}
Peter McMullen.
\newblock Representations of polytopes and polyhedral sets.
\newblock {\em Geometriae Dedicata}, 2:83--99, 1973.

\bibitem[McM96]{McMullen1996}
Peter McMullen.
\newblock Weights on polytopes.
\newblock {\em Discrete Comput. Geom.}, 15(4), 1996.

\bibitem[MPS{\etalchar{+}}09]{MortonPachterShiuSturmfelsWienand2009}
Jason Morton, Lior Pachter, Anne Shiu, Bernd Sturmfels, and Oliver Wienand.
\newblock Convex rank tests and semigraphoids.
\newblock {\em SIAM J. Discrete Math.}, 23(3), 2009.

\bibitem[MUWY18]{MohammadiUhlerWangYu2018}
Fatemeh Mohammadi, Caroline Uhler, Charles Wang, and Josephine Yu.
\newblock Generalized permutohedra from probabilistic graphical models.
\newblock {\em SIAM J. Discrete Math.}, 32(1):64--93, 2018.

\bibitem[Pos09]{Postnikov2009}
Alexander Postnikov.
\newblock Permutohedra, associahedra, and beyond.
\newblock {\em Int. Math. Res. Not. IMRN}, (6):1026--1106, 2009.

\bibitem[PPP23]{PPP2023Nesto}
Arnau Padrol, Vincent Pilaud, and Germain Poullot.
\newblock Deformation cones of graph associahedra and nestohedra.
\newblock {\em European J. Combin.}, 107:No. 103594, 27, 2023.

\bibitem[PPP25]{PPP2023gZono}
Arnau Padrol, Vincent Pilaud, and Germain Poullot.
\newblock Deformed graphical zonotopes.
\newblock {\em Discrete Comput. Geom.}, 73(2):447--465, 2025.

\bibitem[PPPP23]{PadrolPaluPilaudPlamondon}
Arnau Padrol, Yann Palu, Vincent Pilaud, and Pierre-Guy Plamondon.
\newblock Associahedra for finite-type cluster algebras and minimal relations between g-vectors.
\newblock {\em Proceedings of the London Mathematical Society}, 127(3):513--588, 2023.

\bibitem[PRW08]{PostnikovReinerWilliams}
Alexander Postnikov, Victor Reiner, and Lauren~K. Williams.
\newblock Faces of generalized permutohedra.
\newblock {\em Doc.~Math.}, 13:207--273, 2008.

\bibitem[PS19]{PilaudSantos-quotientopes}
Vincent Pilaud and Francisco Santos.
\newblock Quotientopes.
\newblock {\em Bull. Lond. Math. Soc.}, 51(3):406--420, 2019.

\bibitem[{{S}ag}20]{Sage}
The {{S}age Developers}.
\newblock {\em {S}ageMath, the {S}age {M}athematics {S}oftware {S}ystem ({V}ersion 9.1)}, 2020.
\newblock {\tt https://www.sagemath.org}.

\bibitem[Sta07]{Stanley2007}
Richard~P. Stanley.
\newblock An introduction to hyperplane arrangements.
\newblock In {\em Geometric combinatorics}, volume~13 of {\em IAS/Park City Math. Ser.}, pages 389--496. Amer. Math. Soc., Providence, RI, 2007.

\end{thebibliography}
